\documentclass[12pt,authoryear]{article}

\usepackage{color}
\usepackage{dsfont}
\usepackage{enumerate}
\usepackage{graphicx,float}
\usepackage{caption}
\usepackage{subcaption}
\usepackage{setspace}

\usepackage[round]{natbib}
\usepackage{hyperref}

\usepackage{xcolor}
\usepackage{diagbox}
\usepackage{amsmath, amsfonts, amssymb, amsthm, amscd}
\usepackage{mathrsfs}
\usepackage{mathtools} % pour tous les ams[...] - American Mathematical Society

\newtheorem{Theorem}{Theorem}[part]
\newtheorem{Definition}{Definition}[part]
\newtheorem{Proposition}{Proposition}[part]
\newtheorem{Assumption}{Assumption}[part]
\newtheorem{Lemma}{Lemma}[part]

\newtheorem{Remark}{Remark}[part]

\def \N{\mathbb{N}}
\def \R{\mathbb{R}}

\def \E{\mathbb{E}}
\def \F{\mathbb{F}}
\def \G{\mathbb{G}}

\def \d{\mathbf{d}}

\def \f{\mathbf{f}}
\def \x{\mathbf{x}}
\def \y{\mathbf{y}}

\def \P{\mathbb{P}}
\def \Q{\mathbb{Q}}
\def \D{\mathbb{D}}

\def \G{\mathbb{G}}

\def \1{\mathds{1}}
\def \0{\mathds{O}}

\def \Ac{\mathcal{A}}
\def \Bc{\mathcal{B}}
\def \Cc{\mathcal{C}}
\def \Dc{\mathcal{D}}
\def \Ec{\mathcal{E}}
\def \Fc{\mathcal{F}}
\def \Gc{\mathcal{G}}

\def \Kc{\mathcal{K}}
\def \Lc{\mathcal{L}}
\def \Mc{\mathcal{M}}
\def \Nc{\mathcal{N}}

\def \Pc{\mathcal{P}}

\def \Rc{\mathcal{R}}

\def \Lb{\textbf{L}}

\def \Db{\textbf{D}}

\def \Ib{\mathbf{I}}

\def \l{\left}
\def \r{\right}

\newcommand{\restr}[2]{#1_{\mkern 1mu \vrule height 2ex\mkern2mu #2}}

\def\eps{\varepsilon}

\def\beqs{\begin{eqnarray*}}
\def\enqs{\end{eqnarray*}}
\def\beq{\begin{eqnarray}}
\def\enq{\end{eqnarray}}

\providecommand{\keywords}[1]
{
  \small
  \textbf{Keywords:---} #1
}
\providecommand{\MSC}[1]
{
  \small
  \textbf{MSC Classification:--- } #1
}
\title{Controlled superprocesses and HJB equation in the space of finite measures}

\author{Antonio Ocello\\  CMAP, UMR CNRS 7146,\\ École Polytechnique,
\\ antonio.ocello @ polytechnique.edu
}

\begin{document}

\maketitle

\begin{abstract}
    This paper introduces the formalism required to analyze a certain class of stochastic control problems that involve a super diffusion as the underlying controlled system. To establish the existence of these processes, we show that they are weak scaling limits of controlled branching processes. First, we prove a generalized Itô's formula for this dynamics in the space of finite measures, using the differentiation in the space of finite positive measures. This lays the groundwork for a PDE characterization of the value function of a control problem, which leads to a verification theorem. Finally, focusing on an exponential-type value function, we show how a regular solution to a finite--dimensional HJB equation can be used to construct a smooth solution to the HJB equation in the space of finite measures, via the so-called branching property technique.
\end{abstract}

\noindent \MSC{93E20, 49L20, 49L12, 60J70, 35K10, 35Q93, 60H30}

\noindent \keywords{
  Stochastic control, superprocesses, {H}amilton--{J}acobi--{B}ellman equation, dynamic programming principle
}

\section{Introduction}

Superprocesses are a class of stochastic processes that have garnered significant interest in recent years, particularly due to their applications in various fields such as biology and finance.
In population dynamics, they have been introduced as limit approximation for large populations, of branching diffusions, processes used to model the evolution of a population that is dispersed in space \citep[see, $e.g.$,][]{Dawson,Etheridge,Etheridge:Spatial-population-models,li2011measure,Roelly:convergence_measure_val_processes}. Their study has been extensively developed to incorporate increasingly sophisticated model specifications, such as mean-field interactions and varying impacts on model parameters \citep[see, $e.g.$,][]{Perkins,etheridge2004survival,etheridge2023looking,meleard2015stochastic,evans2007damage}.

In \citet{cuchiero2024measure,svaluto2018probability}, these processes have been incorporated into the general framework of measure-valued polynomial processes, a versatile class of stochastic processes known for their tractability and wide applicability in modeling complex systems. In particular, \citet{guida2022measure,cuchiero2022measure} used them for its ability to incorporate stylized facts observed in energy markets.
Additionally, in \citet{cuchiero2019probability}, they focus on the class of probability measure-valued processes, where the Fleming-Viot process serves as a specific example, highlighting the close connections with superprocesses \citep[see, $e.g.$,][]{Dawson,Etheridge}.

The purpose of this article is to introduce and examine the class of \textit{controlled superprocesses}. As highlighted in \citet{etheridge2004survival}, understanding the impact of parameters within this class is crucial for evaluating their effectiveness in critical issues such as survival and avoiding finite-time explosion, which can be approached from an optimization perspective. This class seeks to bridge the gap between spatial optimization at a microscopic level and a broader understanding, linking possible spatial constraints, resource optimization, and general interaction on a global perspective.

This new framework extends classical superprocesses, when the controls are taken to be constant, and they act a proxy for large critical controlled branching diffusion processes, which have recently garnered attention in the literature.
Studied by \citet{Ustunel,Nisio,claisse18,kharroubi2024stochastic}, we are going to build on the description in \citet{Ocello:rel_form_branching}. This relaxed (and weak) control approach, introduced in \citet{ElK:Jeanblanc:Compactification,Haussmann-Existence_Optimal_Controls}, enables us to focus on the laws of these processes associated with starting condition, control, and martingale problem. Following the lines of \citet{Lacker:MFgames:2015,Lacker:limit_theory:2016}, by fixing the first two elements and manipulating the martingale problem, we prove that controlled superprocesses arise as a rescaling of branching processes, thereby establishing their existence.
An example of a study that combines the branching diffusion framework with the mean--field approach is also present in \citet{Claisse:Tan:MFbranching}, where the authors introduce scaling limits that differ from the dynamics of superprocesses.
Moreover, a similar weak reasoning method is also used in \citet{Talbi:Touzi:DPE_MF_Optimal_stopping}, where the optimal stopping problem is generalized to the mean field setting using a control stopping strategy.

We then introduce, to the best of our knowledge, a novel formulation of a weak version of Itô's formula specifically designed for superprocesses. This development is inspired by the methodologies of \citet{martini2023kolmogorov,guo2023ito,cox2024controlled}, where Itô’s formula is initially applied to cylindrical functions and subsequently extended to sufficiently differentiable functions using the denseness of the former class. This approach is advantageous because it relies solely on weak versions of measure flows and it contrasts with methods that use a strong formulation approach \citep[see, $e.g.$,][]{buckdahn2017mean,li2018mean}, the density with respect to a reference measure for the process law \citep{cardaliaguet2019master}, or particle approximation techniques \citep{book:Carmona-Delarue_1,chassagneux2014probabilistic,reis2021relation,dos2023ito}.

Using the Itô's formula, we address a general control problem and establish its well-posedness. This leads to the identification of the value function as a solution to a Hamilton--Jacobi--Bellman (HJB) equation in the space of finite measures and a verification theorem for “classical” solutions, which provides the necessary conditions for proving the optimality of a controlling strategy. This is achieved by showing that the value function of the control problem is a solution of this equation. This point of view has been explored in several works, including those focused on Markovian controls \citet{pham2018bellman}, open-loop controls \citet{bayraktar2018randomized}, Markovian and non-Markovian frameworks \citet{Djete:Possamai:Tan:DPP}, closed-loop controls \citet{Wu:Zhang:ViscosityMFC_closed_loop}, and McKean-Vlasov mixed regular-singular control problems \citet{guo2023ito}.

By employing cost functions that align with the geometry of the dynamics under study, we then provide for an example that effectively utilize the so-called \textit{branching property} technique. This approach has been successfully applied in the literature on controlled branching processes \citep[see, $e.g.$,][]{Nisio,claisse18,kharroubi2024stochastic,kharroubi2024optimal}, where the complex global problem is simplified into a finite-dimensional optimization. This strategy is similar to the approach used in linear--quadratic problems in stochastic control, where solving the initial problem is reformulated into finding specific solutions to a Riccati PDE \citep[see, $e.g.$,][]{abi2021linear,basei2019weak,lefebvre2021linear}.

This work complements the ongoing focus on measure-valued processes in the context of their integration with control theory,  where they have been applied to problems such as mean-field games and mean-field control \citep{Huang:Malhame:Caines:1,Huang:Malhame:Caines:2,huang2007invariance,Huang:Malhame:Caines:4,lasry2006jeux,lasry2006jeux-II,lasry2007mean,Cardaliaguet:NotesMFG,book:Carmona-Delarue_1,book:Carmona-Delarue_2},  controlled measure-valued martingales \citep{cox2024controlled}, and the Zakai equation from filtering theory \citep{martini2023kolmogorov,martini2022kolmogorov}.
While the previous studies focus on different classes of processes with varying state spaces and dynamics, our work builds upon and extends these foundations. The processes considered here, including their state spaces, formulation, and dynamics, are distinct yet complementary to those explored in previous research. This is also evidenced by the differing Itô's formulas provided, which highlight the unique aspects of our approach while integrating and extending the insights from earlier investigations.

The paper is organized as follows: In Section \ref{Section:Setting}, we introduce the model setup, and in Section \ref{Section:def-existence}, we present the controlled superprocesses as a solution to a martingale problem as well as their rescaled counterpart, the class of controlled $n$-rescaled branching diffusion processes. We conclude this section by proving their uniqueness in law and existence as a weak limit of rescaled branching processes.
Section \ref{Section:Ito} focuses on introducing the differential calculus in the space of finite measures and provide the weak version of Itô's formula for the considered dynamics, generalizing the initial martingale problem using the denseness of cylindrical functions in the space of regular functions on finite measures. Finally, Section \ref{Section:ctrl_pb} is centered on the study of the control problem. After establishing a non-explosion bound, we use the Dynamic Programming Principle (DPP) techniques to derive the HJB equation satisfied by the value function of the control problem and prove a verification theorem. We conclude the paper by providing a regular solution to the optimization problem. In the appendix we collect the technical details of the paper. In particular, in \ref{Appendix:proofs} we gather the proofs of the paper's main results, in \ref{Appendix:denseness-cyl-fct} we discuss the denseness properties of cylindrical functions, used to provide Itô's formula, in \ref{Appendix:moment-estimate} we give the moment estimates relying on a martingale measure representation, and in \ref{Appendix:DPP} we give the prove the DPP for our control problem.

\section{Model setup}\label{Section:Setting}

For a Polish space $\l(E,d\r)$ with $\Bc\l(E\r)$ its Borelian $\sigma$-field, we write $C_b\l(E\r)$ (resp. $C_0\l(E\r)$) for the subset of the continuous functions that are bounded (resp. that vanish at infinity), and $\Mc\l(E\r)$ (resp. $\Pc\l(E\r)$) for the set of Borel positive finite measures (resp. probability measures) on $E$. We equip $\Mc\l(E\r)$ with weak* topology, $i.e.$, the weakest topology that makes continuous the maps $\Mc\l(E\r)\ni\lambda\mapsto\int_E \varphi(x) \lambda(dx)$ for any $\varphi\in C_b\l(E\r)$. We denote $\langle \varphi,\lambda\rangle = \int_E \varphi(x) \lambda(dx)$ for $\lambda\in \Mc\l(E\r)$ and $\varphi\in C_b\l(E\r)$.

A family $\mathscr{F}\subseteq C_b\l(E\r)$ is said to be \textit{separating} if, whenever $\langle \varphi,\lambda\rangle=\langle \varphi,\lambda'\rangle$ for all $\varphi\in \mathscr{F}$, and some $\lambda, \lambda'\in \Mc\l(E\r)$, we necessarily have $\lambda=\lambda'$. Since $E$ is Polish, from the Portmanteau theorem \citep[see, $e.g.$, Theorem 1.1.1,][]{SV97}, the set of uniformly continuous functions, for any metric equivalent to $d$, is separating. From Tychonoff's embedding theorem \citep[see, $e.g.$, Theorem 17.8,][]{General:Topology}, $C_b\l(E\r)$ is also separable.
Therefore, there exists a countable and separating family $\mathscr{F}_E = \left\{\varphi_k, k \in\N\right\}$ subset of $C_b\l(E\r)$ such that the function $E\ni x\mapsto 1$ belongs to $\mathscr{F}_E$ and $\|\varphi_k\|_\infty := \sup_E|\varphi_k| \leq 1$ for all $k \in\N$. We use this setting to define the following distance
\beqs
\d_E(\lambda,\lambda') = \sum_{\varphi_k\in\mathscr{F}_E}\frac{1}{2^k}\left|\langle \varphi_k,\lambda\rangle-\langle \varphi_k,\lambda'\rangle\right|,
\enqs
for $\lambda,\lambda'\in \Mc\l(E\r)$.
As in \citet[Theorem 1.1.2,][]{SV97}, this distance $\d_E$ induces on $\Mc\l(E\r)$ the weak* topology. Whenever $E=\R^d$, we adjust this metric to take into account useful differential properties. Let $\mathscr{F}_{\R^d}$ be taken as a subset of $C^2_b\l(\R^d\r)$, the set of real functions with bounded, continuous derivatives over $\R^d$ up to order two. We can take this set as separating since $C^2$ is dense in $C^0$ for local uniform convergence \citep[application of Theorem 8.14,][]{Folland:RealAnalysis}. We define the distance
\beq\label{Superproceq:def_distance_measure}
\d_{\R^d}(\lambda,\lambda') = \sum_{\varphi_k\in\mathscr{F}_{\R^d}}\frac{1}{2^k q_k}\left|\langle \varphi_k,\lambda\rangle-\langle \varphi_k,\lambda'\rangle\right|,
\enq
with $q_k=\max\{1,\|D \varphi_k\|_\infty, \|D^2\varphi_k\|_\infty\}$, and $D$ and $D^2$ denote gradient and Hessian.

\paragraph{Atomic measures}
We write $\Nc^n[E]$ for the space of atomic measures in $E$ where each atom has a mass multiple of $1/n$, $i.e.$,
\beqs
\Nc^n[E] := \left\{\frac{k_i}{n}\sum_{i\in V}\delta_{x_i} ~:~ k_i\in\N,x_i\in E \text{ for }i\in V,~ V\subseteq \N,~|V|<\infty \right\},
\enqs
where $|V|$ is the cardinal of the set $V$. For $n\geq 1$, $\Nc^n[E]$ is a weakly* closed subset of $\Mc\l(E\r)$. If $E$ is a Polish space, $e.g.$ a Euclidean space, $\cup_{n\in\N}\Nc^n[E]$ is dense in $\Mc\l(E\r)$. First, the result is shown for probability measures.
By the fundamental theorem of simulation \citep[see, $e.g.$, Theorem 1.2,][]{pages2018numerical}, there exists a Borel function $\varphi_\lambda : [0,1] \to E$, for any $\lambda\in \Pc\l(E\r)$, such that $\lambda = \restr{\text{Leb}}{[0,1]}\circ \varphi_\lambda^{-1}$,
where $\restr{\text{Leb}}{[0,1]}\circ \varphi_\lambda^{-1}$ denotes the image measure by $\varphi_\lambda^{-1}$ of the Lebesgue measure on the unit interval. With Glivenko–Cantelli theorem  \citep[see, $e.g.$, Theorem 4.1,][]{pages2018numerical}, we approximate the Lebesgue measure on the unit interval by probability measures $\lambda_n\in \Nc^n[E]$. We get the final result decomposing each finite measure $\lambda$ as a probability measure times its total mass $\lambda([0,1])$ and using for the latter the approximation $\lfloor n \lambda\l(E\r)\rfloor / n$, where $\lfloor \cdot \rfloor$ denotes the integer part of a real number.

\paragraph{State space}
Fix a finite time horizon $T > 0$. Let $\Db^d=\D\l(\l[0,T\r];\Mc\l(\R^d\r)\r)$ be the set of càdlàg functions (right continuous with left limits) from $[0, T]$ to $\Mc\l(\R^d\r)$. We endow this space with Skorohod metric $d_{\Db^d}$ associated with the metric $\d_{\R^d}$, which makes it complete
\citep[see, $e.g.$, Theorem 12.2,][]{book:Billingsley}.
For $\P\in \Pc\l(\Db^d\r)$, $\P_t\in \Pc\l(\Mc\l(\R^d\r)\r)$ denotes the time-$t$ marginal of $\P$, $i.e.$, the image of $\P$ under the map $\Db^d\ni \mu \mapsto \mu_t\in \Mc\l(\R^d\r)$. Denote $\Db^{n,d}=\D(\l[0,T\r];\Nc^n\l[\R^d\r])$, a closed subset of $\Db^d$.

We consider the canonical space $\Db^d$, with $\mu$ its canonical process, and $\F^{\mu}=\left\{\Fc_s^{\mu}\right\}_s$ the filtration generated by $\mu$. 
Let $A$ be the set of actions and assume it is a compact subset of $\R^m$. Denote by $\Ac$ the set of $\left\{\Bc\l(\R^d\r)\otimes\Fc^\mu_s\right\}_s$-predictable processes from $\l[0,T\r]\times \R^d$ to $A$, $i.e.$, the set of processes $\alpha$ measurable w.r.t. a $\sigma$-algebra on $\R^d\times\Db^d\times\l[0,T\r]$ generated by all $\left\{\Bc\l(\R^d\r)\otimes\Fc^{\mu}_s\right\}_s$-adapted left-continuous processes when viewed as a mapping $\alpha:\R^d\times\Db^d\times\l[0,T\r] \to\l[0,T\r]$ \citep[see, $e.g.$, Chapter IV.5,][]{revuz2013continuous}.

\paragraph{Model parameters}
Given $d,d'\in\N$, consider the following bounded continuous functions
\beqs
\l(b,\sigma, \gamma
\r) : \R^d\times \Mc\l(\R^d\r)\times A &\to& \R^d \times \R^{d\times d'}\times \R_+
.
\enqs
Suppose $b$ and $\sigma$ are Lipschitz uniformly in $a$, $i.e.$, there exist $C>0$ such that
\begin{align*}
    \begin{split}
        \left|b\l(x,\lambda,a\r)-b\l(x',\lambda',a\r)\right|+\left|\sigma\l(x,\lambda,a\r)-\sigma\l(x',\lambda',a\r)\right|\qquad\qquad
        \\
        \leq
        C \left(\l|x-x'\r| + \d_{\R^d}\l(\lambda,\lambda'\r)\right),
    \end{split}
\end{align*}
for any $x,x'\in\R^d$, $\lambda,\lambda'\in \Mc\l(\R^d\r)$, and $a\in A$.

Let $L$ be the generator defined by
\beqs
L \varphi \l(x,\lambda,a\r) &=& b\l(x,\lambda,a\r)^\top D \varphi(x) + \frac{1}{2}\text{Tr}\left(
\sigma\sigma^\top\l(x,\lambda,a\r)D^2 \varphi(x)
\right),
\enqs
for $\varphi\in C^2_b\l(\R^d\r)$.

\section{The controlled superprocesses}\label{Section:def-existence}
The class of superprocesses is originally introduced as weak limits of branching particle systems \citep[see, $e.g.$,][]{Dawson, Etheridge, Perkins, Roelly:convergence_measure_val_processes}. Consequently, we use the weak control framework outlined in \citet{Ocello:rel_form_branching} to later obtain the scaling limit, which allows us to establish the existence of \textit{controlled superprocesses}.

An alternative methodology for proving the existence of controlled superprocesses involves using the positive maximum principle. This approach, as used in \citet{cuchiero2022measure,cuchiero2024measure,cuchiero2018polynomial,cuchiero2019polynomial}, offers a different perspective compared to the weak limit framework employed in our study. This method offers an alternative to our study, highlighting the diverse strategies available for proving the existence of these processes.

\paragraph{Controlled $n$-rescaled branching diffusion processes}

For $n\in\N$, let $\Lc^n$ be the generator defined on the cylindrical functions $F_\varphi = F(\langle \varphi, \cdot\rangle)$, for $F\in C^{2}_b\l(\R\r)$ and $\varphi\in C^2_{b}\l(\R^d\r)$, as
\begin{align}
\label{Superproceq:Lc_n}
\begin{split}
    &\Lc^n F_\varphi \l(x,\lambda,a\r)
    \\
    &=F'\l(\langle \varphi, \lambda\rangle\r) L \varphi\l(x,\lambda,a\r) + \frac{1}{2n}
    F''\l(\langle \varphi, \lambda \rangle\r)
    \l|D\varphi(x) \sigma\l(x,\lambda,a\r)\r|^2 
    \\
    &\quad
    + \gamma\l(x,\lambda,a\r)n^2\bigg(\frac{1}{2}F\left(\left\langle\varphi,\lambda\right\rangle - \frac{1}{n}\varphi(x)\right)
    \\
    &\phantom{+ \gamma\l(x,\lambda,a\r)n^2\bigg(F}
    +\frac{1}{2}F\left(\left\langle\varphi,\lambda\right\rangle + \frac{1}{n}\varphi(x)\right)
    - F_\varphi \left(\lambda\right)\bigg) .
\end{split}
\end{align}

\begin{Definition}\label{SuperprocDef:branch_diff}
    Fix $\l(t,\lambda_n\r) \in \l[0,T\r]\times \Nc^n\l[\R^d\r]$. We say that $\left(\P, \alpha \right)\in \Pc\l(\Db^d\r)\times \Ac$ is a \emph{controlled} $n$-rescaled branching diffusion process, and we denote $\left(\P, \alpha \right)\in\Rc^n_{\l(t,\lambda\r)}$, if $\P(\mu_t=\lambda_n)=1$ and the process
    \beq
    \label{SuperprocMartPb:diff-F_f-branch_diff}
    M^{F_\varphi,n}_s = F_\varphi\l(\mu_s\r) - \int_t^s \int_{\R^d} \Lc^n F_\varphi \l(x,\mu_r,\alpha_r(x)\r)\mu_r(dx)dr
    \enq
    is a $(\P,\F)$-martingale with quadratic variation
    \begin{align}
    \label{SuperprocMartPb:diff-F_f-branch_diff-quad_var}
    \begin{split}
    \left[M^{F_\varphi,n}\right]_s =& \int_t^s \left(
        F'_\varphi(\mu_r)\right)^2
        \int_{\R^d} \bigg( \frac{1}{n}\left|D\varphi(x)\sigma\l(x,\mu_r,\alpha_r(x)\r)\right|^2+\\
        &\phantom{\int_t^s \left(
            F'_\varphi(\mu_r)\right)^2
            \int_{\R^d} \bigg( \frac{1}{n}}
        \gamma\l(x,\mu_r,\alpha_r(x)\r)\varphi^2(x)
        \bigg) \mu_r(dx)dr,
        \end{split}
    \end{align}
    for any $F\in C^{2}_b\l(\R\r)$, $\varphi\in C^2_{b}\l(\R^d\r)$, and $s\geq t$.
\end{Definition}

\paragraph{Controlled superprocesses}
Let $\Lc$ be the generator defined by
\beqs
\Lc F_\varphi \l(x,\lambda,a\r) &=& F'(\langle \varphi, \lambda\rangle) L \varphi\l(x,\lambda,a\r) + \frac{1}{2}
F''(\langle \varphi, \lambda\rangle)
\gamma\l(x,\lambda,a\r)\varphi^2(x),
\enqs
where $F_\varphi$ denotes the the cylindrical function $F_\varphi= F(\langle \varphi, \cdot\rangle)$, for $F\in C^{2}_b\l(\R\r)$ and $\varphi\in C^2_{b}\l(\R^d\r)$
For simplicity, we write $F'_\varphi(\lambda)$ for $F'(\langle \varphi, \lambda\rangle)$ and $F''_\varphi(\lambda)$ for $F''(\langle \varphi, \lambda\rangle)$.

\begin{Definition}\label{SuperprocDef:Superproc}
Fix $\l(t,\lambda\r) \in \l[0,T\r]\times \Mc\l(\R^d\r)$. We say that $\left(\P, \alpha \right)\in \Pc\l(\Db^d\r)\times \Ac$ is a \emph{
controlled superprocess}, and we denote $\left(\P, \alpha \right)\in\Rc_{\l(t,\lambda\r)}$, if $\P(\mu_t=\lambda)=1$
and the process
    \beq
    \label{SuperprocMartPb:diff-F_f-superprocess}
      M^{F_\varphi}_s = F_\varphi\l(\mu_s\r) - \int_t^s \int_{\R^d} \Lc F_\varphi \l(x,\mu_r,\alpha_r(x)\r)\mu_r(dx)dr
    \enq
    is a $(\P,\F)$-martingale with quadratic variation
    \beq\label{SuperprocMartPb:diff-F_f-quad_var-superprocess}
      \left[M^{F_\varphi}\right]_s = \int_t^s  \left(F'_\varphi(\mu_r)\right)^2\int_{\R^d}\gamma \l(x,\mu_r,\alpha_r(x)\r)\varphi^2(x)\mu_r(dx)dr,
    \enq
    for any $F\in C^{2}_b\l(\R\r)$, $\varphi\in C^2_{b}\l(\R^d\r)$, and $s\geq t$.
\end{Definition}

By taking linear combinations of the functionals $F_\varphi$, we can define the martingale problem for a more extensive class of test functions (see \eqref{SuperprocDef:cyl_functions:time_dip}). This approach is utilized in the proof of Proposition \ref{Prop:existence} (Step 3), where the behavior of this expanded class is leveraged to establish the uniqueness in law for these processes.

\begin{Remark}
    The previous definition represents a generalization of the classical class of superprocesses as described in, $e.g$, \citet{Roelly:convergence_measure_val_processes, Dawson, Etheridge, Perkins,  dynkin1993superprocesses,dynkin2004superdiffusions}, where the Dawson-Watanabe process is an example. In our formulation, the previous class of \textit{superdiffusions} can be retrieved by selecting specific values for the parameters $b$, $\sigma$, and $\gamma$, while setting the control process to a constant.
    
    Moreover, a more general class of controlled superprocesses can be defined with the same martingale problem, as discussed in \citet{etheridge2004survival}, by considering the state space of tempered measures and general $p$-tempered initial conditions, $i.e.$, the space of locally finite measures $\mu$ in $\R^d$ such that
    \begin{align*}
        \int_{\R^d} \frac{1}{\l(1+\l\|x\r\|^2\r)^{p/2}}\mu(dx) <\infty.
    \end{align*}
\end{Remark}

\paragraph{Existence result}
Using the weak formulation for the controlled $n$-rescaled branching diffusion processes allows us to focus on the laws of these processes described by a martingale problem associated with \textit{starting condition} and \textit{control process}. By fixing these elements and manipulating the martingale problem, we prove that controlled superprocesses arise as a rescaling of branching processes, thereby establishing their existence. To achieve this, we extend the Aldous criterion presented in \citet{Dawson, Etheridge, Perkins, Roelly:convergence_measure_val_processes} for convergence to superprocesses to a controlled setting.

Furthermore, we generalize the martingale problem to a class of functionals that are convergence-determining in the space of càdlàg paths on finite measures and then use the ideas of \citet{SV97}, as detailed in \citet{Etheridge, Ocello:rel_form_branching}, to establish uniqueness in law through the duality method.

\begin{Proposition}\label{Prop:existence}
    Fix $t\in\l[0,T\r]$, $\alpha\in \Ac$.
    \begin{enumerate}
        \item For $n\geq1$ and $\lambda_n\in \Nc^n\l[\R^d\r]$, there exists a unique $\P^{t,\lambda_n,\alpha;n}\in \Pc\l(\Db^d\r)$ such that $(\P^{t,\lambda_n,\alpha;n},\alpha)\in\Rc^n_{\l(t,\lambda_n\r)}$.
        \item For $\lambda\in\Mc\l(\R^d\r)$ and a sequence $\left(\lambda_n\right)_{n\in\N}$ such that $\lambda_n\to\lambda$ weakly* and $\lambda_n\in\Nc^n\l[\R^d\r]$, there exists \textcolor{red}{a unique} $\P\in \Pc\l(\Db^d\r)$ such that $(\P,\alpha)\in\Rc_{\l(t,\lambda\r)}$ and we have that $\P^{t,\lambda_n,\alpha;n}\to \P$ for $n\to\infty$.
    \end{enumerate}
\end{Proposition}

A detailed proof of this result can be found in \ref{appendix:proof-existence}.

\section{Itô's formula}
\label{Section:Ito}
Before we proceed to consider the control problem and develop its optimality conditions, we need to examine the behavior of $s\mapsto u\l(\mu_s\r)$ for a certain class of test functions $u$, which is more general than the class of cylindrical functions. To this purpose, we adopt the differential calculus developed in \citet{martini2023kolmogorov}, taking advantage of its denseness results to extend the martingale problem used to introduce these processes. This generalization is possible since the space of finite measures is homeomorphic to a subset of the product between probability measures and the real line when far from the measure zero. This allows, in particular, to use the lifting technique by renormalizing the measures whenever we are not close to this critical measure. 

This approach is similar to the techniques used also by \citet{guo2023ito,cox2024controlled}, where the application of Itô's formula begins with cylindrical functions and is then broadened to encompass sufficiently differentiable functions by leveraging the denseness of the initial class. The strength of this method lies in its reliance exclusively on weak forms of measure flows, making it particularly well-suited for analyzing stochastic processes that are presented in a weak form, as is the case here.

\subsection{Differential calculus}

The growing literature about differential calculus in the space of measures focuses on two main objects: the linear functional derivative (also called flat derivative or extrinsic derivative) and the $L$-derivative (also called intrinsic derivative). The first is defined directly in $\Pc\l(\R^d\r)$, while the second relies on the lifting on a Hilbert space. It is found that one is the spatial gradient of the previous one, coinciding with the notion of derivative of \citet{Lions}, and sometimes this is used as the definition for the $L$-derivatives, like in \citet{cardaliaguet2022splitting}.
Detailed discussions of this topic can also be found for example in \citet{cardaliaguet2019master, book:Carmona-Delarue_1, cardaliaguet2022splitting}.

We present the reformulation of these concepts in $\Mc\l(\R^d\r)$, using the \textit{directional derivative}, which corresponds to the \textit{linear derivative}, as the foundational building block. As also discussed in \citet{martini2023kolmogorov}, this notion is fundamental in both the space of probability measures and the space of finite measures. This approach simplifies the discussion of differential calculus by eliminating the need for more complex notions of differentiation. We refer to \citet{ren2021derivative} to demonstrate that this methodology is equivalent to directly employing more advanced differential concepts.

\begin{Definition}[Linear derivative]
    A function $u: \Mc\l(\R^d\r) \to \R$ is said to have \emph{linear derivative} if it is continuous, bounded and if there exists a function
    \beqs
    \delta_\lambda u : \Mc\l(\R^d\r)\times\R^d\ni (\lambda,x)\mapsto \delta_\lambda u(\lambda,x)\in\R,
    \enqs
    that is bounded, and continuous for the product topology, such that
    \begin{align*}
        u(\lambda) - u(\lambda') = \int_0^1\int_{\R^d} \delta_\lambda u\left(t\lambda +(1-t)\lambda', x \right)
        \l(\lambda-\lambda'\r) \left(dx\right)dt
        ,
    \end{align*}
    for $\lambda, \lambda' \in \Mc\l(\R^d\r)$. We call $C^{1}\l(\Mc\l(\R^d\r)\r)$ the class of functions that are differentiable in linear functional sense from $\Mc\l(\R^d\r)$ to $\R$.
\end{Definition}

Notice that $\delta_\lambda u$ is uniquely defined up to a constant. We take
\beqs
\int_{\R^d}\delta_\lambda u(\lambda,x)\lambda(dx) = 0,
\enqs
as a convention in this paper. Moreover, second-order derivatives are introduced for $u\in C^1\l(\Mc\l(\R^d\r)\r)$ imposing that $\lambda\mapsto \delta_\lambda u(\lambda,x)$  is differentiable in linear functional sense for every $x$ and that $(\lambda,x,y)\mapsto \delta_\lambda^2 u(\lambda,x,y)$ is bounded and continuous. We call $C^{2}\l(\Mc\l(\R^d\r)\r)$ this class of functions.

Finite positive measures could not rely on lifting procedures. For this reason, the notion of \textit{intrinsic} derivative is introduced differentiating with respect to the $x$ component the flat derivative, as done in \citet[Definition 2.2,][]{cardaliaguet2019master}.
\begin{Definition}[Intrinsic derivative]
    Fix $u\in C^1\l(\Mc\l(\R^d\r)\r)$. If $\delta_\lambda u$ is of class $C^1$ with respect to the second variable, the intrinsic derivative $D_\lambda u:\Mc\l(\R^d\r)\times \R^d\to \R$ is
    \beqs
    D_\lambda u(\lambda,x)= \partial_x \delta_\lambda u(\lambda,x).
    \enqs
    We denote with $C^{1,1} \l(\Mc\l(\R^d\r)\r)$ this class of functions.
\end{Definition}

Differentiating with respect to the measure or the space component are two different operations. We denote $C^{k,\ell} \l(\Mc\l(\R^d\r)\r)$ with $k\in N$ to be the collection of functions $u$ that are differentiable $k$ times with respect to the measure and such that the $k$-th derivative with respect to the measure is $\ell$-th times continuously differentiable with respect to its spatial components.

\subsection{Generalized Itô's formula}

We can now present (a weak version of) the Itô's formula for superprocesses. This formula is interpretable only in a weak sense, reflecting the nature of these processes. Let $\Lb$ be the following operator on $u\in C^{2,2}_b\l(\Mc\l(\R^d\r)\r)$
\beqs
    \Lb u(\lambda, x, a)  & := &
    b\l(x,\lambda,a\r)^\top D_\lambda u(\lambda,x) + \frac{1}{2}\text{Tr}\left(
    \sigma\sigma^\top\l(x,\lambda,a\r)\partial_x D_\lambda u(\lambda,x)
    \right)\\
    & & +  \frac{1}{2}\gamma\l(x,\lambda,a\r)\delta_\lambda^2 u(\mu,x,x)
\enqs
for $\l(x,\lambda,a\r)\in \R^d\times \Mc\l(\R^d\r)\times A$.

\begin{Proposition}\label{SuperprocProp:general_mart_pb}
    For $\l(t,\lambda\r) \in \l[0,T\r]\times \Mc\l(\R^d\r)$ and $\alpha\in\Ac$, the following are equivalent:
    \begin{enumerate}[(i)]
        \item $\left(\P^{t,\lambda,\alpha},\alpha\right)\in \Rc_{\l(t,\lambda\r)}$;
        \item the process
        \beq
        \label{SuperprocMartPb:generalized}
        M^{u}_s = u\l(\mu_s\r) - \int_t^s \int_{\R^d} \Lb u\l(x,\mu_r,\alpha_r(x)\r)\mu_r(dx)dr
        \enq
        is a $(\P,\F)$-martingale with quadratic variation
        \beq\label{SuperprocMartPb:generalized-quad_var}
        \left[M^{u}\right]_s = \int_t^s \gamma\l(x,\mu_r,\alpha_r(x)\r)\left|\delta_\lambda u(\mu_r,x)\right|^2  \mu_r(dx)dr,
        \enq
        for any $u\in C^{2,2}_{b}\l(\Mc\l(\R^d\r)\r)$, and $s\geq t$.
    \end{enumerate}
\end{Proposition}
A detailed proof of this result can be found in \ref{appendix:proof-general_mart_pb}.
This result leverages the density of cylindrical functions within $C^{2,2}_{b}\l(\Mc\l(\R^d\r)\r)$-functions, as detailed in \ref{Appendix:denseness-cyl-fct}.

\section{The control problem}\label{Section:ctrl_pb}
We are given two continuous functions $ \psi : \R^d\times \Mc\l(\R^d\r)\times A \to \R$ and $ \Psi : \Mc\l(\R^d\r) \to \R$. We assume that there exists $C>0$ such that
\beq\label{Superproceq:growth_cond_psi_Psi}
\left|\psi\l(x,\lambda,a\r)\right|\leq C\left(1+\d_{\R^d}(\lambda,\0)\right), \qquad
\left|\Psi(\lambda)\right| \leq C\left(1+\d_{\R^d}(\lambda,\0)^2\right)
\enq
for $\l(x,\lambda,a\r)\in \R^d\times \Mc\l(\R^d\r)\times A$ with $\0$ the measure $0$.

Let $J$ and $v$ be respectively the cost and the value functions w.r.t. the controlled superprocesses,
defined as
\begin{align}
\begin{split}
    J(t,\lambda,\alpha) &= \E^{\P^{t,\lambda,\alpha}}\left[\int_t^T \int_{\R^d} \psi(x,\mu_s,\alpha_s(x))\mu_s(dx)ds +\Psi(\mu_T)\right],
\end{split}
\label{eq:cost_fct}
\end{align}
and
\begin{align}
    v\l(t,\lambda\r) =\inf_{\alpha\in\Ac}J(t,\lambda,\alpha),
    \label{eq:value_fct}
\end{align}
for $t\in\l[0,T\r]$, $\lambda\in \Mc\l(\R^d\r)$, and $\alpha\in\Ac$.

\begin{Proposition}[Proposition \ref{SuperprocProp:moment_bounds:mass}]
Fix $\l(t,\lambda\r) \in \l[0,T\r]\times \Mc\l(\R^d\r)$ (resp. $\l(t,\lambda\r) \in \l[0,T\r]\times \Nc^n\l[\R^d\r]$ for $n\geq1$) and $p\in[1,2]$. There exists a constant $C\geq0$, depending only on $T,$ and the coefficient of the parameters, such that
\begin{align*}
    \E^\P\left[\sup_{r\in [t,T]}\d_{\R^d}(\mu_r,\0)^p\right] \leq C~\d_{\R^d}(\lambda,\0)^p,
\end{align*}
for any $\left(\P, \left(\alpha_s\right)_s \right)\in\Rc_{\l(t,\lambda\r)}$.
\end{Proposition}

The previous result entails that the cost functional $J$ is finite for any control $\alpha\in\Ac$. Moreover, using \eqref{Superproceq:growth_cond_psi_Psi}, $J$ is uniformly bounded from below, therefore the optimization problem that defines $v$ is well-posed.

\subsection{Dynamic Programming Principle}\label{SuperprocSection:DPP}

We aim at analyzing the control problem \eqref{eq:cost_fct}-\eqref{eq:value_fct} via a  Dynamic Programming Principle (DPP) approach, a powerful tool for solving control problems.
We present the following result here, but for a more detailed description and comprehensive explanation, please refer to \ref{Appendix:DPP}. This section provides an in-depth analysis and the underlying methodology following the techniques described in \citet{ElK:Tan:Capacities1,ElK:Tan:Capacities2}, ensuring a thorough understanding of the result and its implications.

\begin{Theorem}[Theorem \ref{SuperprocTheorem:DPP}]
    We have
    \begin{equation*}
            v\l(t,\lambda\r) = \inf_{\alpha\in\Ac}\E^{\P^{t,\lambda,\alpha}}\left[
            \int_t^\tau \int_{\R^d} \psi(x,\mu_s,\alpha_s(x))\mu_s(dx)ds + v(\tau,\mu_\tau)\right],
    \end{equation*}
    for any $\l(t,\lambda\r) \in \l[0,T\r]\times \Mc\l(\R^d\r)$, and $\tau$ stopping time taking value in $[t,T]$.
\end{Theorem}

\subsection{HJB Equation}\label{SuperprocSubsection:HJB}

For general control problems, the DPP leads to a characterization of the problem through a nonlinear Hamilton-Jacobi-Bellmann (HJB) equation \citep[see, $e.g.$,][]{Yong-Zhou-StochasticControls}. Looking at \eqref{SuperprocMartPb:generalized}, define the operator $H$ on $\R^d \times \Mc\l(\R^d\r)\times A\times \R^d\times \R^{d\times d} \times \R$ such that
\begin{equation}\label{SuperprocHJB:eq}
    \begin{split}
      H\l(x,\lambda,a,p,M,r\r)=
      b\l(x,\lambda,a\r)^\top p + \frac{1}{2}\text{Tr}\left(
      \sigma\sigma^\top\l(x,\lambda,a\r)M
      \right) \\
      + \frac{1}{2}\gamma\l(x,\lambda,a\r)r +\psi\l(x,\lambda,a\r).
    \end{split}
  \end{equation}
Then, if the value function \eqref{eq:value_fct} is sufficiently smooth, generalizing Proposition \ref{SuperprocProp:general_mart_pb} to function depending in time and measure yields the following HJB equation
\beqs
\begin{cases}
    \partial_t v\l(t,\lambda\r) + \int_{\R^d}\inf_{a\in A}H\Big(x,\lambda,a,D_\lambda v(t,\lambda,x),\\
\qquad\qquad\qquad\qquad\qquad\qquad\qquad\qquad\partial_x D_\lambda v(t,\lambda,x),\delta_\lambda^2 v(t,\lambda,x,x)\Big)
\lambda(dx) = 0,\\
v\l(t,\lambda\r)=\Psi(\lambda).
\end{cases}
\enqs
The form of this HJB equation looks like the one for mean field control \citep[see, $e.g.$,][]{bayraktar2018randomized, Djete:Possamai:Tan:DPP, guo2023ito, pham2018bellman, Wu:Zhang:ViscosityMFC_closed_loop}, where the infimum (or the supremum if maximizing) is taken inside the integral. The major differences here are that we consider the space of finite measures and not only probability measures and that the second-order flat derivative is explicitly involved in the Hamiltonian.

We have the following result in the regular case.

\begin{Theorem}[Verification Theorem]\label{Thm:verification}
    Let $V:\l[0,T\r]\times \Mc\l(\R^d\r)\to \R$ be a function living in $C^{1,(2,2)}_b\l([0,T)\times \Mc\l(\R^d\r)\r)\cap C^{0}\l(\l[0,T\r]\times \Mc\l(\R^d\r)\r)$. 
    \begin{enumerate}[(i)]
        \item Suppose that $V$ satisfies
        \beq\label{SuperprocDPE_for_V:HJB}
            \begin{cases}
            \partial_t V\l(t,\lambda\r)\\
            +\int_{\R^d}\inf_{a\in A}H\Big(x,\lambda,a,D_\lambda V(t,\lambda,x),\\
            \qquad\qquad\qquad\qquad\partial_x D_\lambda V(t,\lambda,x),\delta_\lambda^2 V(t,\lambda,x,x)\Big)
            \lambda(dx) \leq 0,\\
            V\l(t,\lambda\r)\leq\Psi(\lambda).
            \end{cases}
        \enq      
        Then $V\l(t,\lambda\r)\leq v\l(t,\lambda\r)$ for any $\l(t,\lambda\r)\in \l[0,T\r]\times \Mc\l(\R^d\r)$, with $v$ as in \eqref{eq:value_fct}.
        \item  Moreover, suppose that $V$ satisfies \eqref{SuperprocDPE_for_V:HJB} with equality and there exists a continuous function $\hat a(t,x,\lambda)$ for $(t,x,\lambda)\in [0,T)\times\R^d\times \Mc\l(\R^d\r)$, valued in $A$ such that
        \begin{align}
        \label{SuperprocHJB:argmin}
        \begin{split}
            \hat a(t,x,\lambda) \in \arg\min_{a\in A}H\left(x,\lambda,a,D_\lambda v(t,\lambda,x),\qquad\qquad\right.
            \\
            \left.\partial_x D_\lambda v(t,\lambda,x),\delta_\lambda^2 v(t,\lambda,x,x)\right).
        \end{split}
        \end{align}
        Suppose also that the corresponding control \begin{align*}
            \alpha^* = \left\{\alpha^*_s(x) := \hat a(s,x,\mu_s), s\in[t,T)\right\} \in \Ac.
        \end{align*} Then, $V=v$, with $v$ as in \eqref{eq:value_fct}, and $\alpha^*$ is an optimal Markovian control.
    \end{enumerate}
\end{Theorem}

A detailed proof of this result can be found in \ref{appendix:proof-verification}.

\begin{Remark}
    Finding a classical solution to control problems is challenging, even in the finite-dimensional setting, due to the nonlinearity of the operators involved. This difficulty is here further amplified by the infinite-dimensional nature of the space of finite measures. For these reasons, viscosity solutions are typically considered in controlled problems \citep[see, $e.g.$,][]{touzi2012optimal,zhou2024viscosity,burzoni2020viscosity}. However, we leave a detailed investigation of viscosity solutions for future work. We believe that, following the approach in \citet{cox2024controlled}, the value function can be proven to be a viscosity solution to \eqref{SuperprocHJB:eq} and a comparison theorem could be established. Under a continuity condition, one can use the approximation of Dirac masses alongside the standard finite-dimensional comparison theorem and conclude the viscosity characterization.
\end{Remark}

\subsection{Example of regular solution}\label{SuperprocSubsection:branching_property}

We present an example here to demonstrate a scenario in which the HJB equation \eqref{SuperprocHJB:eq} admits a classical solution. Solving this equation in full generality is a complex and challenging problem, particularly due to the infinite-dimensional context.

By using cost functions that align with the dynamics' geometry, this example demonstrates the effective application of a branching property technique. This method, commonly used in the literature on controlled branching processes \citep[see, $e.g.$,][]{Nisio,claisse18,kharroubi2024stochastic,kharroubi2024optimal}, simplifies complex global problems into finite-dimensional optimization tasks.

Assume that there is no dependence on the measure for $b$, $\sigma$, and $\gamma$. With abuse of notation, we denote $b(x,a)$ (resp. $\sigma(x,a)$, $\gamma(x,a)$) instead of $b\l(x,\lambda,a\r)$ (resp. $\sigma\l(x,\lambda,a\r)$, $\gamma\l(x,\lambda,a\r)$). Fix $h\in C_{b}\l(\R^d\r)$ with $h(x)\geq0$ for any $x\in\R^d$, and let $\Psi(\lambda) := \exp\left( - \langle h, \lambda \rangle \right)$ and $\psi=0$. Therefore, the cost function $J$ writes as
\beq
J(t,\lambda,\alpha) &=& \E^{\P^{t,\lambda,\alpha}}\left[\exp( -\langle h, \mu_T \rangle)\right],
\label{eq:cost_fct:exp_case}
\enq
for $\l(t,\lambda\r)\in\l[0,T\r]\times \Mc\l(\R^d\r)$, and $\alpha\in\Ac$.

The following assumption will ensure that there exists a smooth solution to the HJB equation \eqref{SuperprocHJB:eq} associated with this cost function.

\begin{Assumption}\label{SuperprocAssumption:smooth_sol_HJB}
    Assume that the following conditions hold:
    \begin{enumerate} [(i)]
        \item $h\in C^3_b\l(\R^d\r)$;
        \item $\big(b, \sigma, \gamma\big)(\cdot,a)\in C^2\l(\R^d\r)$ for any $a\in A$, and $b$, $\sigma$, and $\gamma$ and their partial derivatives are bounded on $\R^d\times A$;
        \item there exists $C_\sigma>0$ such that
        \beqs
            \sigma\sigma^\top(x,a)\geq C_\sigma \Ib_d, \quad \text{ for }(x,a)\in \R^d\times A.
        \enqs
    \end{enumerate}
\end{Assumption}

\begin{Proposition}\label{Prop:verification:example}
    Suppose that Assumption \ref{SuperprocAssumption:smooth_sol_HJB} holds. Then, there exists a function $w\in C^{1,2}_b\l(\l[0,T\r]\times \R^d\r)$, such that
        \begin{align}
        \label{SuperprocDPE_for_w:exp_case}
            \begin{cases}
                -\partial_t w(t,x)\\
                \quad- \sup_{a\in A}\left\{b(x,a)^\top Dw(t,x) +\displaystyle\frac{1}{2}\text{Tr}\left(\sigma\sigma^\top(x,a)D^2w(t,x)\right) \right.
                \\
                \left.\qquad\qquad\qquad\qquad\qquad\qquad\qquad\qquad\qquad
                -  \displaystyle\frac{1}{2}\gamma(x,a) w(t,x)^2 \right\} = 0,\\
                w(T,x)= h(x).
            \end{cases}
        \end{align}      
        Moreover, we have
        \beqs
            v\l(t,\lambda\r) = \exp\left(\langle w(t, \cdot), \lambda \rangle\right),
        \enqs
        for any $\l(t,\lambda\r)\in \l[0,T\r]\times \Mc\l(\R^d\r)$, with $v$ as in \eqref{eq:value_fct}.
\end{Proposition}

A detailed proof of this result can be found in \ref{appendix:proof-example}.

\section{Conclusion}\label{SuperprocSection:Conclusion}

This article focused on controlled superprocesses, which is, to the best of our knowledge, a novel category of processes. The first part of our study is devoted to introducing the formalism, which we presented in a weak form through a controlled martingale problem. Following this definition, we proved their existence, as weak limits of controlled $n$-rescaled branching processes, as well as their uniqueness in law, given initial condition and control process.

Taking advantage of a differential calculus developped on the space of finite measures, we proved a weak version of the Itô's formula for these dynamics. This was done through the generalization of the original martingale problem to a larger class of functions, applying the denseness properties of the class of cylindrical function. This led to the study of the control problem, establishing the non-explosion property with respect to the chosen distance that metricizes weak* topology. After proving the DPP, we derived an HJB equation on the space of measures and provided a verification theorem. Finally, we used this result to introduce a class of solvable problems. Utilizing the branching property technique, we translated the problem on finite measures to a finite-dimensional nonlinear PDE, reducing the dimensionality. We solved the latter using results from \citet{krylov87}, and, with the help of the verification theorem, we provided an explicit description of the value function.

\appendix
\section{Proofs of main results}\label{Appendix:proofs}

\subsection{Proof of Proposition \ref{Prop:existence}}\label{appendix:proof-existence}

\begin{proof}[\textbf{Step 1.}]
    Fix a control process $\alpha\in \Ac$. For $n=1$, the existence of $\P^{t,\lambda_1,\alpha;1}\in\Pc\l(\Db^d\r)$ such that $(\P^{t,\lambda_1,\alpha;1},\alpha)\in\Rc^1_{\l(t,\lambda^1\r)}$, for any $\l(t,\lambda^1\r) \in \l[0,T\r]\times \Nc^1\l[\R^d\r]$, is discussed in \citet[Section 4,][]{Ocello:rel_form_branching}. This is done for general horizons $T>0$. Existence of $\P^{t,\lambda_n,\alpha;n}\in\Pc\l(\Db^d\r)$ such that $(\P^{t,\lambda_n,\alpha;n},\alpha)\in\Rc^n_{\l(t,\lambda_n\r)}$, for any $\l(t,\lambda_n\r) \in \l[0,T\r]\times \Nc^n\l[\R^d\r]$ stems from this. We denote $\bar\Rc^S_{\l(t,\lambda^1\r)}$ the set of $1$-rescaled branching diffusion processes defined in the interval $[0,S]$, for $S>0$. For $n\in\N$, we define, on the interval $[0,nT]$, the control $\alpha^n$ such that
    \beqs
    \alpha^n_s = \alpha_{s/n}, \quad \text{ for }s \in [0,nT].
    \enqs
    Fix $\l(t,\lambda_n\r) \in \l[0,T\r]\times \Nc^n\l[\R^d\r]$. From the previous result, we have the existence of $\P^{n,1}\in\Pc\l(\D\l([0,nT];\Mc\l(\R^d\r)\r)\r)$ such that $(\P^{n,1},\alpha^n)\in\bar \Rc^{nT}_{(nt,n\lambda_n)}$. Define the map $\mathfrak{R}^n$ such that 
    \beqs
    \mathfrak{R}^n: \D(\l[0,T\r];\Nc^n\l[\R^d\r]) &\to& \D([0,nT];\Nc^1\l[\R^d\r])\\
    \left(\mu_s\right)_{s\in\l[0,T\r]} &\mapsto&  \left(n\mu_{s/n}\right)_{s\in[0,nT]}.
    \enqs
    As in \citet{Dawson,Etheridge,Roelly:convergence_measure_val_processes}, we have that $\P^{t,\lambda_n,\alpha;n}:=\P^{n,1}\circ(\mathfrak{R}^n)^{-1}\in\Pc\l(\Db^d\r)$ is such that $(\P^n, \alpha)\in \Rc^{n}_{\l(t,\lambda_n\r)}$.

    To establish uniqueness in law, we follow the lines of \citet[Proposition 4.5,][]{Ocello:rel_form_branching}. This involves extending the arguments used in the proposition the rescaled martingale problem \eqref{SuperprocMartPb:diff-F_f-branch_diff}-\eqref{SuperprocMartPb:diff-F_f-branch_diff-quad_var}.

\textbf{\textit{Step 2.}}
    Fix a control process $\alpha\in\Ac$. Consider a sequence $\left(\lambda_n\right)_{n\in\N}$ such that $\lambda_n\to\lambda$ weakly* and $\lambda_n\in\Nc^n\l[\R^d\r]$.
    From \textbf{\textit{Step 1}}, there exists $\P^n\in \Pc\l(\Db^d\r)$ such that $(\P^n,\alpha)\in\Rc^n_{\l(t,\lambda_n\r)}$. Our goal is to show $\left(\P^n\right)_{n\in\N}$ converges weakly to some $\P\in \Pc\l(\Db^d\r)$ and that $(\P,\alpha)\in\Rc_{\l(t,\lambda\r)}$.
    
    We define the projection $\pi_\varphi$ as
    \beqs
        \pi_\varphi:\Mc\l(\R^d\r)\ni \lambda\mapsto\langle \varphi, \lambda\rangle\in \R,
    \enqs
    for any $f\in C^0\l(\R^d\r)$. Clearly, the weak* topology is the weakest topology for which the mappings $\pi_{\varphi_k}$ are continuous, for $\left\{\varphi_k\right\}_{k\geq1}$ dense in $C_0\l(\R^d\r)$. Moreover, under $\P^n$, for the semimartingale $\langle\varphi_k,\mu_\cdot\rangle$, we define the predictable finite variation process as $V^n_\cdot(\varphi)$ and the increasing process of the martingale part as $I^n_\cdot(\varphi)$ for $k\geq 1$.
    From equations \eqref{SuperprocMartPb:diff-F_f-branch_diff} and \eqref{SuperprocMartPb:diff-F_f-branch_diff-quad_var}, we have 
    \beq\label{Superproceq:expression_Vn}
    V^n_s(\varphi_k) &=& \int_t^s \int_{\R^d} L \varphi_k\l(x,\mu_r,\alpha_r(x)\r) \mu_r(dx)dr,\\
    I^n_s(\varphi_k) &=& \int_t^s \int_{\R^d} \bigg( \frac{1}{n} \left|D\varphi_k(x)\sigma\l(x,\mu_r,\alpha_r(x)\r)\right|^2\label{Superproceq:expression_In}\\
    &&\phantom{\int_t^s \int_{\R^d} \bigg( \frac{1}{n}D\varphi_k(x)\sigma}
    +\gamma\l(x,\mu_r,\alpha_r(x)\r)\varphi_k^2(x)\bigg)\mu_r(dx)dr.\nonumber
    \enq
    We can now verify conditions \textit{(i)} and \textit{(ii)} of \citet[Theorem 2.3,][]{Roelly:convergence_measure_val_processes} to prove that $\left\{\P^n\right\}_n$ is tight. This means proving that
    \begin{enumerate} [(i)]
        \item $\left(\P^n\circ \pi_{\varphi_k}^{-1}\right)_{n\geq1}$ is tight for $k\geq1$;
        \item $V^n_s(\varphi_k)$ and $I^n_s(\varphi_k)$ satisfy the following condition of Aldous for any $k\geq 1$: for each stopping time $\tau$ we can find a sequence $\delta_n$ such that $\delta_n\to0$ as $n\to \infty$ and such that
        \begin{equation}\label{Superproceq:conv_Vn_In}
          \begin{split}
            \limsup_n\E^{\P_n}\left[\left|V^n_{\tau+\delta_n}(\varphi) - V^n_{\tau}(\varphi)\right|\right] = 0, \\ \limsup_n\E^{\P_n}\left[\left|I^n_{\tau+\delta_n}(\varphi) - I^n_{\tau}(\varphi)\right|\right] = 0.
          \end{split}
        \end{equation}
    \end{enumerate}
    
    From \eqref{SuperprocMartPb:diff-F_f-branch_diff}, we have that $\langle 1,\mu_\cdot \rangle$ is a $\P^n$-martingale for any $n\geq 1$. Therefore, for any $n\geq1$,
    \beqs
    \P^n\left(\sup_{s\in[t,T]}\langle 1,\mu_s\rangle>K\right)
    \leq \frac{1}{K} \E^{\P^n}\left[\langle 1,\mu_t\rangle\right]=\frac{1}{K} \langle 1,\lambda_n\rangle.
    \enqs
    Since $\lim_n\langle 1,\lambda_n\rangle=\langle 1,\lambda\rangle$, we obtain that $\sup_n \P^n\left(\sup_{s\in[t,T]}\langle 1,\mu_s\rangle>K\right)$ tends to $0$ when $K$ tends to infinity. $\left(\P^n\circ \pi_{\varphi_k}^{-1}\right)_{n\geq1}$ is also tight for $k\geq1$, since each function of $C_0\l(\R^d\r)$ is bounded. Therefore, (i) is satisfied.
    
    Fix $\varphi\in\{1\}\cup\left\{\varphi_k:k\geq1\right\}$, a stopping time $\tau$ taking values in $[t,T]$, and $\delta_n>0$.
    Using \eqref{Superproceq:expression_Vn} and \eqref{Superproceq:expression_In}, we have
    \begin{align*}
        \E^{\P_n} \left[\left|V^n_{\tau+\delta_n}(\varphi) - V^n_{\tau}(\varphi)\right|\right]
        &\leq
        \E^{\P_n} \left[\int_\tau^{\tau+\delta_n} \int_{\R^d} |L \varphi(x,\mu_r, \alpha_r(x))| \mu_r(dx)dr
        \right]
        \\
        &\leq
        \delta_n \E^{\P_n}\left[\langle 1,\mu_\tau\rangle\right]\;
        \l\|L \varphi\r\|_\infty = \delta_n \langle 1,\lambda_n\rangle \|L \varphi\|_\infty,
    \end{align*}
    where the last equality comes from the martingale property. By the same arguments, we also have
    \begin{align*}
        &\E^{\P_n}\l[
        \l|I^n_{\tau+\delta_n}(\varphi) - I^n_{\tau}(\varphi)\r|
        \r] \\
        &\quad
        =\E^{\P_n}\l[
            \int_\tau^{\tau+\delta_n} \int_{\R^d} \l(
                \frac{1}{n} \left|D\varphi(x)\sigma\l(x,\mu_r,\alpha_r(x)\r)\r|^2
                \r.\r.
        \\
        &\quad
        \l.\l.
        \phantom{ \E^{\P_n}\int_\tau^{\tau+\delta_n} \int_{\R^d} \frac{1}{n}\qquad\qquad\qquad\qquad
        }
            +\gamma\l( x,\mu_r,\alpha_r(x) \r)\varphi^2(x)\right)\mu_r(dx)dr\r]
        \\
        &\quad
    \leq \delta_n \langle 1,\lambda_n\rangle\l( \|D\varphi\sigma\|^2_\infty + 2\bar \gamma \|\varphi\|^2_\infty \r).
    \end{align*}

    Therefore, if $\lim_n\delta_n=0$, we get \eqref{Superproceq:conv_Vn_In}, which gives that $\left(\P_n\right)_{n\geq1}$ is tight in $\Db^d$.
    
    To conclude, we take a sequence $\left(\P_n\right)_{n\geq1}$ converging to a probability measure $\P \in \Pc\l(\Db^{d}\r)$ and prove that $(\P,\alpha)\in\Rc_{\l(t,\lambda\r)}$.
    To do that, we focus on the convergence of $\Lc_n$. For $(x,\nu,a)\in \R^d \times \Mc\l(\R^d\r)$, the third term in the expression of $\Lc_n$ in \eqref{Superproceq:Lc_n} is equal to
    \begin{align*}
        W_n(x,\nu,a) &=
        \gamma(x,\nu,a)\;n^2
        \l(
            \frac{1}{2}
            F \l( \l\langle\varphi,\nu\r\rangle - \frac{1}{n}\varphi(x)\r)
        \r.
        \\
        &\l.\phantom{=
        \gamma(x,\nu,a)\;n^2\frac{1}{2}\quad
        }
        +\frac{1}{2} F\l(\l\langle\varphi,\nu\r\rangle
            + \frac{1}{n}\varphi(x)\r)
        - F\l(\langle \varphi, \nu\rangle\r)\r).
    \end{align*}
    Using Taylor's development with Lagrange reminder, we have 
    \beqs
        W_n(x,\nu,a) =
        \gamma(x,\nu,a)\frac{
            F''\left(\left\langle\varphi,\nu\right\rangle + z^n_1\right)+F''\left(\left\langle\varphi,\nu\right\rangle + z^n_2\right)}{2},
    \enqs
    with $z^n_1$ (resp. $z^n_2$) a point in $\{h\langle\varphi,\nu\rangle+(1-h)\varphi(x)/n: h\in [0,1]\}$ (resp. $\{h\langle\varphi,\nu\rangle-(1-h)\varphi(x)/n: h\in [0,1]\}$). Since $\gamma$ is bounded, we have $W(x,\nu,a)=\lim_n W_n(x,\nu,a) = F''_\varphi(\nu)\gamma(x,\nu,a)\varphi^2(x)$ for any $(x,\nu,a)$. We can now prove that $F_\varphi(\mu_\cdot) - \int_t^\cdot \int_{\R^d} \Lc F_\varphi \l(x,\mu_r,\alpha_r(x)\r)\mu_r(dx)dr$ is a $\P$-martingale, $i.e.$, for each stopping time $\tau$ taking value in $[t,T]$,
    \beqs
    \E^{\P}\left[
    F_\varphi(\mu_\tau) - F_\varphi(\mu_t) - \int_t^\tau \int_{\R^d} \Lc F_\varphi \l(x,\mu_r,\alpha_r(x)\r)\mu_r(dx)dr
    \right] = 0.
    \enqs
    We have
    \beqs
    \lim_n \E^{\P_n}\left[
    F_\varphi(\mu_\tau) - F_\varphi(\mu_t)
    \right] = \E^{\P}\left[
    F_\varphi(\mu_\tau) - F_\varphi(\mu_t)
    \right],
    \enqs
    and
    \begin{align*}
        &\E^{\P}\left[\int_t^\tau \int_{\R^d} \Lc F_\varphi \l(x,\mu_r,\alpha_r(x)\r)\mu_r(dx)dr \right] 
        \\
        &\qquad\qquad\qquad\qquad\qquad-
        \E^{\P_n}\left[\int_t^\tau \int_{\R^d} \Lc_n F_\varphi \l(x,\mu_r,\alpha_r(x)\r)\mu_r(dx)dr \right]
        \\
        & = \E^{\P}\left[\int_t^\tau \int_{\R^d} \Lc F_\varphi \l(x,\mu_r,\alpha_r(x)\r)\mu_r(dx)dr \right] 
        \\
        &\qquad\qquad\qquad\qquad\qquad
        - \E^{\P_n}\left[\int_t^\tau \int_{\R^d} \Lc F_\varphi \l(x,\mu_r,\alpha_r(x)\r)\mu_r(dx)dr \right] \\
        &\qquad\qquad\qquad\qquad\qquad
        +\E^{\P_n}\left[\int_t^\tau \int_{\R^d} (\Lc-\Lc_n) F_\varphi \l(x,\mu_r,\alpha_r(x)\r)\mu_r(dx)dr \right].
    \end{align*}
    The last term on the right side satisfies
    \beqs
        &&\l| \E^{\P_n}\left[\int_t^\tau \int_{\R^d} (\Lc-\Lc_n) F_\varphi \l(x,\mu_r,\alpha_r(x)\r)\mu_r(dx)dr \right]\r|
        \\
        &&=\l| \E^{\P_n}\l[\int_t^\tau \int_{\R^d} \l(
            \frac{1}{2 n}
            F''\l(\l\langle \varphi, \mu_r\r\rangle\r)
            \l| D\varphi(x) \sigma\l(x,\mu_r,\alpha_r(x)\r)\r|^2 +
            \r.\r.\r.
        \\
        &&
        \l.\l.\l.
        \phantom{
        \E^{\P_n}[\int_t^\tau \int_{\R^d} (\frac{1}{2 n}
        \qquad
        }
        W\l(x,\mu_r,\alpha_r(x)\r) - W_n\l(x,\mu_r,\alpha_r(x)\r)\r)
        \mu_r(dx)dr \r]\r|\\
        &&\leq \frac{C}{n}\l(1+ T \l\langle 1, \lambda_n\r\rangle\r),
    \enqs
    for a constant $C$ which depends only on $F''_\varphi, \sigma, D\varphi, \gamma, \varphi$. Hence,
    \beqs
    &&\lim_n\E^{\P_n}\left[
    F_\varphi(\mu_\tau) - F_\varphi(\mu_t) - \int_t^\tau \int_{\R^d} \Lc_n F_\varphi \l(x,\mu_r,\alpha_r(x)\r)\mu_r(dx)dr
    \right] \\
    &&\quad=\E^{\P}\left[
    F_\varphi(\mu_\tau) - F_\varphi(\mu_t) - \int_t^\tau \int_{\R^d} \Lc F_\varphi \l(x,\mu_r,\alpha_r(x)\r)\mu_r(dx)dr
    \right] = 0.
    \enqs

\textbf{\textit{Step 3.}}
    Now, we focus on the uniqueness in law for the martingale problems \eqref{SuperprocMartPb:diff-F_f-branch_diff}-\eqref{SuperprocMartPb:diff-F_f-branch_diff-quad_var} and \eqref{SuperprocMartPb:diff-F_f-superprocess}-\eqref{SuperprocMartPb:diff-F_f-quad_var-superprocess}.

    Using \textit{Doob's functional representation theorem} \citep[see, $e.g.$, Lemma 1.13,][]{book:KALLENBERG-FMP}, we remark that a $\left\{ \Bc\l(\R^d\r)\otimes\Fc_s\right\}_s$-predictable process $\alpha$ from $\l[0,T\r]\times \R^d$ to $A$ boils down to be a predictable map $\mathfrak{a}$ such that
    \begin{align}
    \label{Superproceq:doob_repr_ctrl}
    \begin{split}
        \mathfrak{a} : \l[0,T\r]\times\R^d\times\Db^d &\to A\\
        \left( s,x,\left(\mu_r\right)_{r\in\l[0,T\r]} \right) &\mapsto \mathfrak{a}\left(s,x,\left(\mu_{r\wedge s}\right)_{r\in[0,T]}\right) = \alpha_s(x).
    \end{split}
    \end{align}
    The reason that each $\alpha$ can be interpreted as a map from $[0,T] \times \mathbb{R}^d \times \mathbb{D}^d$ to $A$ is that $\mathcal{F}_s$ is generated by $\left(\mu_r\right)_{r \in [0,s]}$. Furthermore, for a fixed $s\in [0,T]$, by applying the monotone class theorem \citep[see, e.g., Chapter 0.2,][]{revuz2013continuous} and the definition of $\left\{\mathcal{B}(\mathbb{R}^d) \otimes \mathcal{F}_s \right\}_s$-predictable processes, we can observe that the realization of this map depends only on $\left(\mu_{r \wedge s}\right)_{r \in [0,T]}$.
    
    The martingale problem \eqref{SuperprocMartPb:diff-F_f-superprocess}-\eqref{SuperprocMartPb:diff-F_f-quad_var-superprocess} (resp. \eqref{SuperprocMartPb:diff-F_f-branch_diff}-\eqref{SuperprocMartPb:diff-F_f-branch_diff-quad_var}) can be easily generalized to a domain that characterizes the law of processes in $\l[0,T\r]\times \Db^d$ \citep[as done in ][]{Ocello:rel_form_branching}. First, we introduce the domain of cylindrical functions $\Dc\subseteq C^0(\l[0,T\r]\times\Db^d)$ as the set of $F_{\left(f_1, \dots, f_p\right)}: [0, T]\times \Db^d \to \R$ of the form
    \beq\label{SuperprocDef:cyl_functions:time_dip}
    F_{\left(f_1, \dots, f_p\right)}(s,\x) = F\left( \langle f_1(s\wedge t_1,\cdot), \x_{s\wedge t_1}\rangle,\ldots,\langle f_p(s\wedge t_p,\cdot), \x_{s\wedge t_p}\rangle\right),
    \enq
    for $(s,\x)\in \R_+\times \Db^{d}$ and some $p\geq 1$, $t_1,\ldots,t_p \in \l[0,T\r]$, $F\in C^{2}_b\l(\R^p\r)$, and $f_1,\ldots,f_p\in C^{1,2}_b\l(\l[0,T\r]\times \R^{d}\r)$. For $f\in C^{1,2}_b\l(\l[0,T\r]\times \R^{d}\r)$ , we use the notation $L f(s,x,\lambda,a) = L f(s,\cdot)\l(x,\lambda,a\r)$. For a measurable function $\beta :\l[0,T\r]\times \R^d\to A$, we then define the operator $\mathbb{L}^\beta$ (resp. $\mathbb{L}^{\beta,n}$) on $\Dc$ by
    \begin{align*}
        &\mathbb{L}^{\beta} F_{\left(f_1, \dots, f_p\right)}(s,\x)  
        \\
        &
        =D F\left( \langle f_1(s\wedge t_{1},\cdot), \x_{s\wedge t_1}\rangle,\ldots,\langle f_p(s\wedge t_{p},\cdot), \x_{s\wedge t_p}\rangle\right)^\top {\mathfrak{L}}^{\beta}\f(s,\x)
        \\
        &\quad +  \frac{1}{2}\textrm{Tr}\Big( \left\langle \mathfrak{S}^{\beta}\f (\mathfrak{S}^{\beta}\f)^\top (s,\cdot), \x_s \right\rangle
        D^2F\big( \langle f_1(s\wedge t_{1},\cdot), \x_{s\wedge t_1}\rangle,\ldots,
        \\
        &
        \phantom{
        +  \frac{1}{2}\textrm{Tr}\Big( \left\langle \mathfrak{S}^{\beta}\f (\mathfrak{S}^{\beta}\f)^\top (s,\cdot), \x_s \right\rangle D^2F(\dots)\ldots,\qquad
        }
        \langle f_p(s\wedge t_{p},\cdot), \x_{s\wedge t_p}\rangle\big)\Big)
    \end{align*}
    \begin{align*}
        &\text{(resp.}
        \\
        &\mathbb{L}^{\beta,n} F_{\left(f_1, \dots, f_p\right)}(s,\x)  
        \\
        &
        =D F\left( \langle f_1(s\wedge t_{1},\cdot), \x_{s\wedge t_1}\rangle,\ldots,\langle f_p(s\wedge t_{p},\cdot), \x_{s\wedge t_p}\rangle\right)^\top {\mathfrak{L}}^{\beta}\f(s,\x)
        \\
        &\quad +  \frac{1}{2n}\textrm{Tr}\Big( \left\langle \mathfrak{S}^{\beta,n}\f (\mathfrak{S}^{\beta,n}\f)^\top (s,\cdot), \x_s \right\rangle
        D^2F\big( \langle f_1(s\wedge t_{1},\cdot), \x_{s\wedge t_1}\rangle,\ldots,
        \\
        &
        \phantom{
            +  \frac{1}{2}\textrm{Tr}\Big( \left\langle \mathfrak{S}^{\beta}\f (\mathfrak{S}^{\beta}\f)^\top (s,\cdot), \x_s \right\rangle D^2F(\dots)\ldots,\qquad
        }
        \langle f_p(s\wedge t_{p},\cdot), \x_{s\wedge t_p}\rangle\big)\Big)
        \\
        &
        \quad+\sum_{j=1}^p\mathds{1}_{t_{j-1}<s\leq t_j} \int_{\R^d}n^2 ~\gamma\l(x,\beta(s,x), \x_s\r)
        \\
        & \qquad \l(  \frac{1}{2}F\l(  \langle f_1(s\wedge t_1,\cdot), \x_{s\wedge t_1} \rangle,\ldots, \langle f_{j-1}(s\wedge t_{j-1},\cdot), \x_{s\wedge t_{j-1}}\rangle,\r.\r.
        \\
        &
        \phantom{
            + \f (\mathfrak{S}^{\beta}\f)^\top (s,\cdot), \x_s DF(\dots)\ldots,
        }
        \left.
        \mathfrak{G}^{n,+}_{k} f_j(s,x,\x_{s}), \ldots,\mathfrak{G}^{n,+}_{k} f_p(s,x,\x_{s})\right)
        \\
        & \qquad  ~+\frac{1}{2}F\left(  \langle f_1(s\wedge t_1,\cdot), \x_{s\wedge t_1} \rangle,\ldots, \langle f_{j-1}(s\wedge t_{j-1},\cdot), \x_{s\wedge t_{j-1}}\rangle,\r.
        \\
        &
        \phantom{
            + \f (\mathfrak{S}^{\beta}\f)^\top (s,\cdot), \x_s DF(\dots)\ldots,
        }
        \left.
        \mathfrak{G}^{n,-}_{k} f_j(s,x,\x_{s}),\ldots,\mathfrak{G}^{n,-}_{k} f_p(s,x,\x_{s})\right)\\
      &  \qquad \qquad \qquad  -F\left( \langle f_1(s\wedge t_1,\cdot), \x_{s\wedge t_1}\rangle,\ldots,\langle f_p(s\wedge t_p,\cdot), \x_{s\wedge t_p}\rangle\right)\bigg)\x_s(dx)
      \text{ ),}
    \end{align*}
    with $t_0=0$, where 
    \begin{align*}
        \mathfrak{L}^{\beta} \f(s,\x)
        &:=
        \left(\begin{array}{c}
            \mathds{1}_{s\leq t_1}
            \int_{\R^d}\partial_t f_1(s,x)+ L f_1(s,x,\x_{s},\beta(s,x)) \x_{s}(dx)\\ \vdots\\
            \mathds{1}_{s\leq t_p}
            \int_{\R^d}\partial_t f_p(s,x)+ L f_p(s,x,\x_{s},\beta(s,x)) \x_{s}(dx)
        \end{array}\right),
        \\
        \mathfrak{S}^{\beta} \f(s,x,\x)
        &:=
        \left(\begin{array}{c}
            \mathds{1}_{s\leq t_1} f_1(s,x)\sqrt{\gamma\l(x,\x_{s},\beta(s,x)\r)}
            \\ \vdots\\
            \mathds{1}_{s\leq t_p} f_p(s,x)\sqrt{\gamma\l(x,\x_{s},\beta(s,x)\r)}
        \end{array}\right),
        \\
        \mathfrak{S}^{\beta,n} \f(s,x,\x)
        &:=
        \left(\begin{array}{c}
            \mathds{1}_{s\leq t_1}\l|D f_1(s,x)\sigma\l(x,\x_{s},\beta(s,x)\r)\r|
            \\ \vdots\\
            \mathds{1}_{s\leq t_p} \l|D f_p(s,x)\sigma\l(x,\x_{s},\beta(s,x)\r)\r|
        \end{array}\right),
        \\
        \mathfrak{G}^{n,\pm}_{k} f_j (s,x,\x)
        &:=
        \langle f_j(s,\cdot), \x_{s}\rangle \pm\frac{1}{n}f_j (s,x),
    \end{align*}
    for $(s,x,\x)\in \l[0,T\r]\times \R^d\times \Db^d$. Following the language of \citet{EthierKurtz}, we call the graph of $\Dc$ the \textit{full} generator $\G$ (resp. $\G^n$), with
    \beq\label{Superproceq:def_full_generator}
        \G:= \left\{(g, \mathbb{L}^{\cdot}g): g\in \Dc\right\}
        \qquad
        \G^n:= \left\{(g, \mathbb{L}^{\cdot,n}g): g\in \Dc\right\}
        .
    \enq
    
    We define the domain $\Dc^T\subseteq C^0\l(\Mc\l(\R^d\r)\r)$ of the functions
    \beq\label{Superproceq:def_D_T}
        F_{\left(f_1, \dots, f_p\right)}(\lambda) = F\left( \langle f_1, \lambda\rangle,\ldots,\langle f_p, \lambda\rangle\right),\quad \text{ for } \lambda\in \Mc\l(\R^{d}\r),
    \enq
    for some $p\geq 1$, $F\in C^{2}_b\l(\R^p\r)$, and $f_1,\ldots,f_p\in C^{1,2}_b\l(\l[0,T\r]\times \R^{d}\r)$. These functions are embedded in $\Dc$ when we consider functions as in \eqref{SuperprocDef:cyl_functions:time_dip} such that $f_i$ does not depend on $s$ and $t_i=T$, for $i=1,\dots,p$. Therefore, with abuse of notation, we say that $\mathbb{L}$ (resp. $\mathbb{L}^{\cdot,n}$) acts on $\Dc^T$ with the obvious adjustments.
    
    Considering the canonical process $\mu\in \Db^{d}$, we have that, if $(\P,\alpha)\in \Rc_{(t, \lambda)}$ the process 
    \beq\label{SuperprocMartPb:diff-F_f-extended}
    \bar M^h_s:=h\l(\mu_{s\wedge\cdot}\r)-\int_t^s \mathbb{L}^{\alpha} h\l(\mu_{r\wedge\cdot}\r)dr,\quad \text{ for } t\leq r\leq T,
    \enq
    is a $(\P,\F)$-martingale with quadratic variation equal to
    \begin{align}
    \label{SuperprocMartPb:diff-F_f-extended-quad_var}
    \begin{split}
        \left[\bar M^h\right]_s :=& \int_t^s \textrm{Tr}\Big( \l\langle \mathfrak{S}^{\alpha}\f (\mathfrak{S}^{\alpha}\f)^\top (r,\cdot), \mu_r \r\rangle
        \\
        &
        \quad
        +DF(DF)^\top
        \big(
        \l\langle
            f_1\l(r\wedge t_{1},\cdot\r), \mu_{r\wedge t_1}\r
        \rangle,
        \ldots,
        \\
        &
        \qquad\qquad\qquad\qquad\qquad\qquad
        \l\langle
            f_p\l(r\wedge t_{p},\cdot\r), \mu_{r\wedge t_p} 
        \r\rangle\big)\Big)dr,
    \end{split}
    \end{align}
    for any $h = F_{\left(f_1, \dots, f_p\right)} \in \Dc$. Therefore, we can finally prove uniqueness in law given a fixed initial condition and control for the controlled superprocesses as follows.
    
    Therefore, following \citet[Proposition 4.5 and Theorem 4.1,][]{Ocello:rel_form_branching}, which generalize \citet[Theorem 4.1 and Theorem 4.2,][]{EthierKurtz}, there exists at most one $\P^{t,\lambda,\alpha}\in\Pc\l(\Db^d\r)$ (resp. $\P^{t,\lambda_n,\alpha;n}\in\Pc\l(\Db^d\r)$) such that $\l(\P^{t,\lambda,\alpha},\alpha\r)\in\Rc_{\l(t,\lambda\r)}$  (resp. $\l(\P^{t,\lambda_n,\alpha;n},\alpha\r)\in\Rc^n_{\l(t,\lambda_n\r)}$) for a given initial condition $\l(t,\lambda\r) \in \l[0,T\r]\times \Mc\l(\R^d\r)$ (reps. $\l(t,\lambda_n\r) \in \l[0,T\r]\times \Nc^n\l[\R^d\r]$) and control $\alpha\in\Ac$.
\end{proof}

\subsection{Proof of Proposition \ref{SuperprocProp:general_mart_pb}}\label{appendix:proof-general_mart_pb}.

\begin{proof}[\nopunct]
    $(ii)\implies(i)$: From Equation \eqref{eq:derivative_cyl_fct}, $F_\varphi\in C^{2,2}_{b}\l(\R^d\r)$ for any $F\in C^2_b\l(\R\r)$ and $\varphi\in C^2_b\l(\R^d\r)$ and equations \eqref{SuperprocMartPb:generalized} and \eqref{SuperprocMartPb:generalized-quad_var} becomes \eqref{SuperprocMartPb:diff-F_f-superprocess} and \eqref{SuperprocMartPb:diff-F_f-quad_var-superprocess} for $u=F_\varphi$.

    $(i)\implies(ii)$: If $\lambda=\0$, it is clear that $M^{u}$ is constant in time, thus a martingale with a null quadratic variation.
    Consider a starting condition $\l(t,\lambda\r)$, with $\langle1,\lambda\rangle>0$, a control $\alpha\in \Ac$ and the sequence of stopping times $\{\tau_k\}_{k\geq 1}$ as
    \beqs
        \tau_k :=
        \inf\left\{s\geq t: \langle1,\mu_s\rangle>k\right\}
        \wedge
        \inf\left\{s\geq t: \langle1,\mu_s\rangle<\frac{1}{k}\right\}.
    \enqs
    Defining $\mu^k_\cdot:=\mu_{\cdot\wedge\tau_k}$, for $\langle1,\lambda\rangle\in [1/k,k]$, we have that under $\P^{(t,\lambda,\alpha)}$, this process lives in $\Mc_k\l(\R^d\r)$ by construction. Thus, applying Lemma \ref{SuperprocLemma:density:K_k_N} and Lemma \ref{SuperprocLemma:density:D_T}, there exists a sequence $u_n\in \Dc^T$ such that $u_n \to u$ as $n \to \infty$ pointwise as well as their derivatives and there exists $C>0$ such that $\|u_n\|_{C^{2,2}_b\l(\Mc\l(\R^d\r)\r)}\leq C \|u\|_{C^{2,2}_b\l(\Mc\l(\R^d\r)\r)}$.

    Equation \eqref{eq:derivative_cyl_fct} shows how derivatives operate on cylindrical functions. Looking at \eqref{SuperprocMartPb:diff-F_f-extended} and \eqref{SuperprocMartPb:diff-F_f-extended-quad_var}, we see that for $h\in\Dc^T$ equations \eqref{SuperprocMartPb:generalized} and \eqref{SuperprocMartPb:generalized-quad_var} are satisfied if applied to $\mu^k$. We prove now that
    $$
        u(\mu^k_\cdot) - \int_t^\cdot \int_{\R^d} \Lb u \l(x,\mu_r,\alpha_r(x)\r)\mu_r(dx)dr
    $$ is a $\P^{(t,\lambda,\alpha)}$-martingale, $i.e.$, for each stopping time $\theta$ in $[t,T]$,
    \beqs
    \E^{\P^{(t,\lambda,\alpha)}}\left[
        u(\mu^k_\theta) - u(\lambda) - \int_t^\theta \int_{\R^d} \Lb u (x,\mu^k_r,\alpha_r(x))\mu^k_r(dx)dr
    \right] = 0.
    \enqs
    Since \eqref{SuperprocMartPb:generalized} is satisfied for $h\in \Dc^T$, and from the bounds on the derivatives and on the coefficients $b$, $\sigma$ and $\gamma$, we can apply the Dominated Convergence Theorem and obtain 
    \beqs
    0 &=& \lim_n \E^{\P^{(t,\lambda,\alpha)}}\left[
        u_n\left(\mu^k_\theta\right) - u_n(\lambda) - \int_t^\theta \int_{\R^d} \Lb u_n \left(x,\mu^k_r,\alpha_r(x)\right)\mu^k_r(dx)dr
    \right]\\
    &=& \E^{\P^{(t,\lambda,\alpha)}}\left[
        u\left(\mu^k_\theta\right) - u(\lambda) - \int_t^\theta \int_{\R^d} \Lb u\left(x,\mu^k_r,\alpha_r(x)\right)\mu^k_r(dx)dr
    \right].
    \enqs
    By definition of the quadratic variation and \eqref{SuperprocMartPb:generalized-quad_var} applied to $u_n\in \Dc^T$, we have for $n\in \N$ that
    \begin{equation}
        \label{SuperprocMartPb:generalized-quad_var:u_n}
          \begin{split}
            &\E^{\P^{(t,\lambda,\alpha)}}\left[\left(
                u_n\left(\mu^k_s\right) - u_n(\lambda) - \int_t^s \int_{\R^d} \Lb u_n \left(x,\mu^k_r,\alpha_r(x)\right)\mu^k_r(dx)dr
                \right)^2\right]\\
            &= \E^{\P^{(t,\lambda,\alpha)}}\left[
            \int_t^s \gamma\left(x,\mu^k_r,\alpha_r(x)\right)\left|\delta_\lambda u_n\left(\mu^k_r,x\right)\right|^2  \mu^k_r(dx)dr
          \right].
        \end{split}
      \end{equation}
    Therefore, we apply again Dominated Convergence Theorem again and obtain \eqref{SuperprocMartPb:generalized-quad_var:u_n} with respect to $u$. 
    
    Finally, we can remove the localization using the Dominated Convergence Theorem since $u\in C^{2,2}_{b}\l(\R^d\r)$ and the bound \eqref{Superproceq:moment_bound:mass}.
\end{proof}

\subsection{Proof of Theorem \ref{Thm:verification}}\label{appendix:proof-verification}
\begin{proof}[\nopunct]

\textit{(i)}
    Since $V\in C^{1,(2,2)}_b\l([0,T)\times \Mc\l(\R^d\r)\r)$, we have for all $\l(t,\lambda\r)\in [0,T)\times \Mc\l(\R^d\r)$, and $\alpha\in\Ac$, by \eqref{SuperprocMartPb:generalized}, the process
    \beqs
    V(s,\mu_s) - V\l(t,\lambda\r) - \int_t^s \partial_t V(u,\mu_r) + \int_{\R^d} \Lb V(u,x,\mu_r,\alpha_r(x))\mu_r(dx)dr
    \enqs
    is a martingale under $\P^{t,\lambda,\alpha}$. By taking the expectation, we get
    \beqs
        &&\E^{\P^{t,\lambda,\alpha}} \left[V(s,\mu_s)\right] 
        \\
        &&= V\l(t,\lambda\r) + \E^{\P^{t,\lambda,\alpha}} \left[\int_t^s \partial_t V(u,\mu_r) + \int_{\R^d} \Lb V(u,x,\mu_r,\alpha_r(x))\mu_r(dx)dr\right].
    \enqs
    Since $V$ satisfies \eqref{SuperprocDPE_for_V:HJB}, we have
    \begin{align*}
        \begin{split}
            \partial_t V(u,\mu_r) + \int_{\R^d} \Lb V(u,x,\mu_r,\alpha_r(x))+ \psi\l(x,\mu_r,\alpha_r(x)\r)\mu_r(dx) \leq 0,
            \\ \P^{t,\lambda,\alpha}-\text{a.s.}
        \end{split}
    \end{align*}
    for any $\alpha\in \Ac$. Therefore,
    \begin{align*}
        \begin{split}
            \E^{ \P^{t,\lambda,\alpha} } \left[V(s,\mu_s)\right] \leq V\l(t,\lambda\r) - \E^{\P^{t,\lambda,\alpha}} \left[\int_t^s \int_{\R^d} \psi\l(x,\mu_r,\alpha_r(x)\r)\mu_r(dx)dr\right],
            \\ \P^{t,\lambda,\alpha}-\text{a.s.}
        \end{split}
    \end{align*}
    for any $\alpha\in \Ac$. Since $V$ is continuous on $\l[0,T\r]\times \Mc\l(\R^d\r)$, we obtain by the dominated convergence theorem and \eqref{SuperprocDPE_for_V:HJB}
    \begin{align*}
        \begin{split}
            \E^{\P^{t,\lambda,\alpha}} \left[\Psi(\mu_T)\right] \leq V\l(t,\lambda\r) - \E^{\P^{t,\lambda,\alpha}} \left[\int_t^T \int_{\R^d} \psi\l(x,\mu_r,\alpha_r(x)\r)\mu_r(dx)dr\right],
            \\
            \P^{t,\lambda,\alpha}-\text{a.s.}
    \end{split}
    \end{align*}
    for any $\alpha\in \Ac$. From the arbitrariness of the control, we deduce that $V\l(t,\lambda\r)\leq v\l(t,\lambda\r)$, for all $\l(t,\lambda\r)\in\l[0,T\r]\times \Mc\l(\R^d\r)$.

\textit{(ii)}
    By \eqref{SuperprocMartPb:generalized},
    \begin{align*}
        &\E^{\P^{t,\lambda,\alpha}} \left[V(s,\mu_s)\right]
        \\
        &= V\l(t,\lambda\r) + \E^{\P^{t,\lambda,\alpha}} \left[\int_t^s \partial_t V(u,\mu_r) + \int_{\R^d} \Lb V(u,x,\mu_r,\alpha_r(x))\mu_r(dx)dr\right].
    \end{align*}
    By definition of $\hat a(t,x,\lambda)$, we have
    \beqs
    \partial_t V\l(t,\lambda\r) + \int_{\R^d}\Lb V(t,\lambda, x,\hat a(t,x,\lambda)) +\psi(x,\lambda,\hat a(t,x,\lambda))\lambda(dx) = 0,
    \enqs
    and so
    \beqs
        \E^{ \P^{t,\lambda,\alpha^*} } \left[V(s,\mu_s)\right] = V\l(t,\lambda\r) - \E^{\P^{t,\lambda,\alpha}} \left[
            \int_t^s \int_{\R^d}\psi(x,\lambda,\alpha^*_r(x,\mu_r))\mu_r(dx)dr
        \right].
    \enqs
    By sending $s$ to $T$ , we then obtain
    \beqs
        V\l(t,\lambda\r) = \E^{\P^{t,\lambda,\alpha^*}} \left[\Psi(\mu_T)+ \int_t^T \int_{\R^d}\psi(x,\lambda,\alpha^*_r(x,\mu_r))\mu_r(dx)dr\right] = J(t,\lambda,\alpha^*),
    \enqs
    which shows that $V\l(t,\lambda\r)=J(t,\lambda,\alpha^*)\geq v\l(t,\lambda\r)$. Therefore, $V=v$ and $\alpha^*$ is an optimal Markovian control.
\end{proof}

\subsection{Proof of Proposition \ref{Prop:verification:example}} \label{appendix:proof-example}
\begin{proof}[\nopunct]
    Our goal is to determine the function $w$ using \citet[Theorem 6.4.4,][]{krylov87}, which guarantees the existence of smooth solutions for a certain category of fully nonlinear partial differential equations. First, we need to modify \eqref{SuperprocDPE_for_w:exp_case} to fall into this class. If $w$ satisfies \eqref{SuperprocDPE_for_w:exp_case}, we see that the function $\tilde w(t,x):= e^{-t}w(t,x)$ satisfies the following nonlinear PDE
    \beq\label{SuperprocDPE_for_w_tilde:exp_case}
        \begin{cases}
            -\partial_t \tilde w(t,x) - \sup_{a\in A}\bigg\{b(x,a)^\top D\tilde w(t,x) +\displaystyle\frac{1}{2}\text{Tr}\left(\sigma\sigma^\top(x,a)D^2\tilde w(t,x)\right)\\
            \phantom{-\partial_t w(t,x) - \sup_{a\in A}\bigg\{ }
            - \frac{1}{2}\gamma(x,a) e^t \tilde w(t,x)^2 + \tilde w(t,x)\bigg\} = 0,
            \\
            \tilde w(T,x)= e^{-T} h(x).
        \end{cases}
    \enq
    Let $C_\gamma>0$ be a constant such that
    \beqs
    \gamma(x,a)\geq C_\gamma, \quad \text{ for all }(x,a)\in \R^d\times A.
    \enqs
    Without loss of generality, we can take $C_\gamma$ to be such that $\frac{C_\gamma e^T}{4}>1$. This means that
    \beqs
        &&\frac{1}{2}\gamma(x,a) e^t M_0^2 - M_0\leq \frac{C_\gamma e^T}{2}M_0^2 - M_0 \leq -\delta_0,\\
        &&\frac{1}{2}\gamma(x,a) e^t M_0^2 + M_0\geq M_0 \geq \delta_0,
    \enqs
    for all $(x,a)\in \R^d\times A$, with $M_0:= \frac{4}{C_\gamma e^T}>0$ and $\delta_0:= M_0^2$. Combining these inequalities with Assumption \ref{SuperprocAssumption:smooth_sol_HJB}, the property that equation \eqref{SuperprocDPE_for_w_tilde:exp_case} belongs to the class of \citet[Theorem 6.4.4,][]{krylov87} follows from \citet[Example 6.1.8,][]{krylov87}. Therefore, this theorem ensures that there exists $\tilde w\in C^{1,2}_b\l(\l[0,T\r]\times \R^d\r)$ solution to \eqref{SuperprocDPE_for_w_tilde:exp_case}. Thus, $w(t,x):= e^{t}\tilde w(t,x)$ is a bounded solution of \eqref{SuperprocDPE_for_w:exp_case} belonging to $ C^{1,2}_b\l(\l[0,T\r]\times \R^d\r)$.

    We conclude this proposition by applying Theorem \ref{Thm:verification}. Define $V\l(t,\lambda\r):= \exp\left(\langle w(t, \cdot), \lambda \rangle\right)$. Using the terminal condition of $w$, we see that $V\l(t,\lambda\r)=\Psi(\lambda)$. Moreover, \eqref{SuperprocHJB:eq} in this setting writes as
    \begin{align*}
        &\exp\l(\l\langle w(t, \cdot), \lambda \r\rangle\r)\int_{\R^d} \l(-\partial_t w(t,x) \phantom{\frac{1}{2}}\r.
        \\
        &\qquad\qquad\qquad\qquad- \sup_{a\in A}\l\{b(x,a)^\top D w(t,x)
        +\frac{1}{2}\text{Tr}\l(\sigma\sigma^\top(x,a)D^2 w(t,x)\r)
         \r.
         \\
         &\qquad\qquad\qquad\qquad\qquad\qquad\qquad\qquad\qquad\l.\l.
        - \frac{1}{2}\gamma(x,a) w(t,x)^2\r\}\r)\lambda(dx) = 0.
    \end{align*}
    It is now clear that $V$ satisfies \eqref{SuperprocHJB:eq} since $w$ satisfies \eqref{SuperprocDPE_for_w:exp_case}. Moreover, the optimal control $\hat a$, defined as in \eqref{SuperprocHJB:argmin}, is the point that reaches the maximum over a compact set of a continuous function. This means that $\hat a$ can be chosen continuously, thus predictably. Therefore, its associated optimal control belongs to $\Ac$. We can now apply Theorem \ref{Thm:verification} and conclude.
\end{proof}

\section{Denseness properties of the cylindrical functions}\label{Appendix:denseness-cyl-fct}

This section is a reminder of fundamental theorems that show the denseness properties of the class of cylindrical functions and follows the procedure outlined in \cite{martini2023kolmogorov}, which shares similar techniques to the ones outlined in \citet{guo2023ito,cox2024controlled}. First, we restrict ourselves in a compact space, prove the denseness of $\Dc^T$ as defined in \eqref{Superproceq:def_D_T} in $C^{2,2}\l(\Mc\l(\R^d\r)\r)$ and conclude with a localization argument.

In particular, as in \citet[Example 2.9,][]{martini2023kolmogorov}, we have that $\Dc^T\subseteq C^{2,2}\l(\Mc\l(\R^d\r)\r)$, where $\Dc^T$ is the domain of cylindrical functions as in \eqref{Superproceq:def_D_T}.
    In particular, it holds that
    \begin{align}\label{eq:derivative_cyl_fct}
    \begin{split}
        \delta_\lambda h(\lambda,x) &= DF\left( \langle f_1,\lambda\rangle,\ldots,\langle f_p, \lambda\rangle\right)^\top \f(x),
        \\
        \delta_\lambda^2 h(\lambda,x,y) &= \f(y)^\top DF^2\left( \langle f_1,\lambda\rangle,\ldots,\langle f_p, \lambda\rangle\right) \f(x),
        \\
        D_\lambda h(\lambda,x) &= DF\left( \langle f_1,\lambda\rangle,\ldots,\langle f_p, \lambda\rangle\right)^\top D\f(x),
    \end{split}
    \end{align}
    with $\f(x) := \left(f_1, \dots, f_p\right)(x)^\top$, and $D\f(x) = \left(Df_1, \dots, Df_p\right)(x)^\top$, for  $h = F_{\left(f_1, \dots, f_p\right)} \in \Dc^T$.

\paragraph{A compact set}
We work into a compact set. This is done to apply the Stone-Weierstrass theorem and prove the denseness of cylindrical functions. Consider the set of compact rectangles $\left\{K_N:=[-N,N]^d\right\}_{N\geq 1}\subseteq\R^d$. For every $k,N\in\N$, we define the
\beqs
    \Mc_k\l(\R^d\r)&:=& \left\{
    \lambda\in \Mc\l(\R^d\r):\lambda\l(\R^d\r)\in \left[\frac{1}{k},k\right]\right\},\\
    \Kc^k_N &:=& \left\{
    \lambda\in \Mc_k\l(\R^d\r):\text{supp}(\lambda)\subseteq K_N\right\}.
\enqs
These sets are non-empty. In particular, $\Kc^k_N$ is compact for the weak* topology for any $k,N\in \N$. Indeed, this space is homeomorphic to $\Kc^{k,1}_N\times \left[\frac{1}{k},k\right]$ with $\Kc^{k,1}_N:= \Kc^k_N\cap \Pc\l(\R^d\r)$ with the homeomorphism
\beqs
    \mathfrak{H}:\Mc\l(\R^d\r)\backslash \{\0\}&\to& \Pc\l(\R^d\r)\times \R_+\\
    \lambda &\mapsto& \left(\frac{1}{\lambda\l(\R^d\r)} \lambda,\lambda\l(\R^d\r)\right)
    .
\enqs
The set $\Kc^{k,1}_N$ is weakly* precompact, using Prokhorov's theorem \citep[see, $e.g.$, Theorem 1.6.1,][]{book:Billingsley}, and is closed as any limit point of the sequence in $\Kc^{k,1}_N$ also has support contained in $K_N$. Therefore, $\Kc^k_N$ is compact as homeomorphic to the product of two compact sets.

\paragraph{Cutting-off strictly positive measures}
Given $\lambda\in \Mc_k\l(\R^d\r)$, we denote with $\rho^N \mu$ the measure such that $\frac{d\rho^N\lambda}{d\lambda} = \rho^N$, with $\rho^N$ a positive and smooth cut-off function equal to $1$ in $K_N$ and
identically zero outside $K_{N+1}$. We observe that $\rho^N\lambda$ is always in $\Kc^k_{N+1}$. Thus, we set
\beq\label{Superproceq:approx_r}
u^N(\lambda) := u(\rho^N \lambda), \quad \text{ for } \lambda\in \Mc_k\l(\R^d\r), N \geq 1.
\enq
With these notations, we give \citet[Lemma 3.3,][]{martini2023kolmogorov} that proves the first approximation theorem for functions on $\Mc_k\l(\R^d\r)$.

\begin{Lemma}\label{SuperprocLemma:density:K_k_N}
    Fix $k\geq 1$ and $u\in C^{2,2}\l(\Mc_k\l(\R^d\r)\r)$. Let $\big\{u^N\big\}_{N\geq 1}$ be the sequence defined by \eqref{Superproceq:approx_r}. Then, $u^N(\lambda) \to u(\lambda)$ as $n\to\infty$, for every $\lambda\in \Mc_k\l(\R^d\r)$, and we have that the sequences $\big\{\delta_\lambda u^N\big\}_{N\geq 1}$, $\big\{\delta_\lambda^2 u^N\big\}_{N\geq 1}$, $\big\{D_\lambda u^N\big\}_{N\geq 1}$, and $\big\{\partial_x D_\lambda u^N\big\}_{N\geq 1}$ pointwise converge to the respective derivatives of $u$. Moreover, $\|u^N\|_\infty \leq \|u\|_\infty$, and there exists $C > 0$ independent of $u$, $N$, and $k$ such that
    \beqs
        \|\delta_\lambda u^N\|_\infty &\leq& \|\delta_\lambda u\|_\infty,\\
        \|\delta_\lambda^2 u^N\|_\infty &\leq& \|\delta_\lambda^2 u\|_\infty,\\
        \|D_\lambda u^N\|_\infty &\leq& C\left(\|D_\lambda u\|_\infty +\|\delta_\lambda u\|_\infty\right),
        \\
        \|\partial_x D_\lambda u^N\|_\infty &\leq& C\left(\|D_\lambda u\|_\infty +\|\delta_\lambda u\|_\infty + \|\partial_x D_\lambda u\|_\infty \right).
    \enqs
\end{Lemma}

This result allows us to approximate functions in $C^{2,2}\l(\Mc_k\l(\R^d\r)\r)$ with functions in $C^{2,2}\l(\Kc^k_N\r)$, for $k,N\in\N$. We can now adapt \citet[Lemma 3.4,][]{martini2023kolmogorov} showing that the domain of cylindrical functions $\Dc^T$ is dense in $C^{2,2}\l(\Kc^k_N\r)$.

\begin{Lemma}\label{SuperprocLemma:density:D_T}
    Fix $k\geq 1$ and $N\geq 1$. Let $u$ be in $C^{2,2}\l(\Kc^k_N\r)$. There exists a sequence of cylindrical functions $\left\{u_n\right\}_{n\geq1}\subseteq \Dc^T$ such that $u_n(\lambda) \to u(\lambda)$ as $n \to \infty$, for every $\lambda\in\Mc\l(\R^d\r)$, and $\left\{\delta_\lambda u_n\right\}_{n\geq1}$, $\left\{\delta_\lambda^2 u_n \right\}_{n\geq1}$, $\left\{D_\lambda u_n \right\}_{n\geq1}$, and $\left\{\partial_x D_\lambda u_n \right\}_{n\geq1}$ converge pointwise to the respective derivatives of $u$ for any $\lambda\in \Kc^k_N$.  Moreover, $\| u_n\|_\infty \leq \|u\|_\infty$,  and the same holds for the derivatives, up to a multiplicative constant independent of $u$, $N$ and $k$.
\end{Lemma}

Since $\Kc^k_N$ is compact in $\Mc\l(\R^d\r)$, the previous lemma proves 
\beqs
\|u_n - u\|_{C^{2,2}_b\l(\Mc\l(\R^d\r)\r)}\to 0, \quad \text{as } n\to \infty,
\enqs
where
\begin{equation}\label{Superprocnorm_C22b}
    \begin{split}
      \|u\|_{C^{2,2}_b\l(\Mc\l(\R^d\r)\r)} := 
      \sup_{\lambda\in \Kc^k_N,x,y\in K_N}\Big\{
      |u(\lambda)| + |\delta_\lambda u(\lambda,x)| + 
      |\delta_\lambda^2 u(\lambda,x,y)| \qquad \quad
      \\
      + |D_\lambda u(\lambda,x)|
      + |\partial_x D_\lambda u(\lambda,x)|\Big\},
    \end{split}
  \end{equation}
for $u$ in $C^{2,2}\l(\Mc\l(\R^d\r)\r)$ which is bounded with bounded derivatives. A stronger norm could be used by adding in the supremum also the terms depending on $\partial_x \delta_\lambda^2 u$ and $D_\lambda^2u$  as in \citet{martini2023kolmogorov}. 
For our scope, norm \eqref{Superprocnorm_C22b} is enough to generalize the martingale problem \eqref{SuperprocMartPb:diff-F_f-superprocess}.

\begin{Remark}
    A different approach could have been taken to prove that $\Dc^T$ is dense in $C^{2,2}\l(\Kc^k_N\r)$. As proven in \citet{guo2023ito}, this domain separates points in $\Kc^k_N$ and vanishes at no point, therefore using Stone-Weierstrass theorem, it is dense in the $C^{0}\l(\Kc^k_N\r)$ with the topology of the strong convergence. Then, using the definition for the Linear derivative and the intrinsic derivative, the convergences of the different derivatives could be established, as in \citet[Lemma 3.12,][]{guo2023ito}.
\end{Remark}

\section{Moment estimates}\label{Appendix:moment-estimate}

We now consider the representation of the controlled superprocesses as Stochastic Differential Equations (SDE), similar to \citet{Ocello:rel_form_branching}. This will help in providing moment estimates for these processes and proving that the control problem \eqref{Section:ctrl_pb} is well posed.

This representation makes use of martingale measures, in extensions of the original space, and lets us apply the general theory of semimartingales in a more general setting. Relevant definitions and results on these objects are concisely summarised in \citet{ElK_Meleard:mart_measure} \citep[see, $e.g.$,][]{Walsh} for a monograph on the subject). We recall briefly their definition.

\begin{Definition}
Let $(G,\Gc)$ be a Lusin space with its $\sigma$-algebra, and $(
\Omega, \Fc, \P,\F =\left\{\Fc_s\right\}_s)$ a filtered space satisfying the usual condition, where we define $\Pc$ the predictable $\sigma$-field. A process $\mathfrak{M}$ on $\l[0,T\r]\times \Omega \times\Gc$ is called  \emph{martingale measure on $G$} if
\begin{enumerate}[(i)]
\item $\mathfrak{M}_0\l(\Ec\r)=0$ a.s. for any $\Ec\in\Gc$;
\item $\mathfrak{M}_t$ is a $\sigma$-finite, $L^2\l(\Omega\r)$-valued measure for all $t\in \l[0,T\r]$;
\item $\left(\mathfrak{M}_t\l(\Ec\r)\right)_{t\in\l[0,T\r]}$ is an $\F$-martingale for any $\Ec\in\Gc$.
\end{enumerate}
We say that $\mathfrak{M}$ is \emph{orthogonal} if the product $\mathfrak{M}_t\l(\Ec\r)\mathfrak{M}_t\l(\Ec'\r)$ is a martingale for any two disjoint sets $\Ec,\Ec'\in \Gc$. We also say, on one hand, that it is \emph{continuous} if $\left(\mathfrak{M}_t\l(\Ec\r)\right)_{t\geq0}$ is continuous, \emph{purely discontinuous}, on the other hand, if $\left(\mathfrak{M}_t\l(\Ec\r)\right)_{t\geq0}$ is a purely discontinuous martingale for any $\Ec\in \Gc$.

Finally, if $\mathfrak{M}$ is continuous, we call \emph{intensity} the $\sigma$-finite positive measure $\nu\l(ds, de\r)$ on $\l[0,T\r]\times \Omega \times\Gc$, $\F$-predictable, such that for each $\Ec\in\Gc$ the process $\l(\nu\l(\l(0,t\r]\times \Ec \r)\r)_{t\geq0}$ is predictable, and corresponds to the quadratic variation of $\left(\mathfrak{M}_t\l(\Ec\r)\right)_{t\geq0}$, for any $\Ec\in\Gc$ and $t\in\l[0,T\r]$.

\end{Definition}

\begin{Proposition}\label{SuperprocProp:representation}
Let $\left(\P, \left(\alpha_s\right)_s \right)\in\Rc_{\l(t,\lambda\r)}$. There exists an extension $\Big( \hat\Omega=\Db^d\times\tilde\Omega, \hat\Fc=\Fc^\mu_T\otimes\tilde\Fc, \hat\P=\P\otimes\tilde\P,\left\{\hat\Fc_s=\Fc^\mu_s\otimes\tilde\Fc_s\right\}_s\Big)$ of $\left(\Db^d, \Fc^\mu_T, \P,\F^\mu\right)$, where we naturally extend $\mu$ and $\alpha$, that satisfies the following properties.
\begin{enumerate}
    \item $(\hat\Omega,\hat\Fc,\hat\F,\hat\P)$ is a filtered probability space supporting a continuous $\hat\F$-martingale measures $\mathfrak{M}$ on $\l[0,T\r]\times \hat\Omega \times \Bc\l(\R^d\r)$, with intensity measure $\mu_r(dx)dr$.
    \item $\hat\P\circ X_t^{-1}=\lambda$.
    \item We have that
    \begin{align}\label{SuperprocSDE:weak}
        \begin{split}
            \langle f,\mu_s\rangle = \langle f,\mu_t\rangle &+ \int_t^s\int_{\R^d} L f \l(x,\mu_r,\alpha_r(x)\r)\mu_r(dx)dr\\
            &+ \int_t^s\int_{\R^d} \sqrt{\gamma\l(x,\mu_r,\alpha_r(x)\r)} f(x) \mathfrak{M}(dx,dr)~.
        \end{split}
    \end{align}
    for all $f\in C^\infty_b\l(\R^d\r)$ and all $s\in[t,T]$.
\end{enumerate}
\end{Proposition}

\begin{proof}
    The representation of these processes is grounded in representation theorems for continuous martingale measures. We follow \citet{Roelly-Meleard:discont_measure_val_branching} and \citet[Proposition 3.3,][]{Ocello:rel_form_branching} applying their construction here.
\end{proof}

We can now prove the non-explosion of these processes, which will imply the well-posedness of the optimization problem.
\begin{Proposition}\label{SuperprocProp:moment_bounds:mass}
Fix $\l(t,\lambda\r) \in \l[0,T\r]\times \Mc\l(\R^d\r)$ and $p\in[1,2]$. There exists a constant $C\geq0$, depending only on $T,$ and the coefficient of the parameters, such that
\beq
\label{Superproceq:moment_bound:mass}
    \E^\P\left[\sup_{r\in [t,T]}\d_{\R^d}(\mu_r,\0)^p\right] &\leq& C\d_{\R^d}(\lambda,\0)^p,
\enq
for any $\left(\P, \left(\alpha_s\right)_s \right)\in\Rc_{\l(t,\lambda\r)}$, where $\0$ denotes the measure $0$.
\end{Proposition}
\begin{proof}
    Fix $\left(\P, \left(\alpha_s\right)_s \right)\in\Rc_{\l(t,\lambda\r)}$. We recall that
    $$
        \d_{\R^d}(\mu_r,\0) = \sum_{\varphi_k\in\mathscr{F}_{\R^d}}\frac{1}{2^k q_k}\left|\langle \varphi_k,\mu_r\rangle\right|,
    $$
    for any $r\in[t,T]$. We define the stopping times $\tau_N$ as 
    \beqs
    \tau_N = \inf \left\{ u\geq t : 
    \langle 1,\mu_r\rangle \geq N
    \right\},
    \enqs
    and denote $\mu^{N}_s:=\mu_{\tau_N\wedge s}$, for $N\geq1$. Proposition \ref{SuperprocProp:representation} implies that there exists an extension of $\Omega$ where $\mu$ can be satisfies
    \eqref{SuperprocSDE:weak} on the stochastic interval $[t, \tau_N]$. Such SDE is driven by $\mathfrak{M}^{N}$, a orthogonal continuous martingale measure in $\l[0,T\r]\times\R^d$, with the intensity measure $\mu_s(dx)\1_{s\leq \tau_N}ds$. Applying \eqref{SuperprocSDE:weak} to $\varphi_k$, we have 
    \begin{align*}
        \langle \varphi_k,\mu_s^N\rangle = \langle \varphi_k,\lambda\rangle &+\int_t^s\int_{\R^d} L \varphi_k \l(x,\mu_r,\alpha_r(x)\r)\mu_r(dx)\1_{r\leq \tau_N}dr\\
        &+ \int_t^s\int_{\R^d}
            \sqrt{\gamma\l(x,\mu_r,\alpha_r(x)\r)} \varphi_k(x) \mathfrak{M}^{N}(dx,dr).
    \end{align*}
    for $s\geq t$, and $k\in\N$.
    Applying Young's inequality, there is a constant $C$ (which may change from line to line) such that
    \begin{align*}
        &\E^\P\left[\sup_{s\in [t,T]}|\langle \varphi_k,\mu_s^N\rangle|^p\right]
        \\
        &\leq C|\langle \varphi_k,\lambda\rangle|^p 
        +
        C\E^\P\Bigg[\sup_{s\in [t,T]}\Bigg|
        \int_t^s\int_{\R^d} L \varphi_k \l(x,\mu_r,\alpha_r(x)\r)\mu_r(dx)\1_{r\leq \tau_N}dr
        \Bigg|^p\Bigg]
        \\
        &\quad
        +C\E^\P\left[\sup_{s\in [t,T]}\left|\int_t^s\int_{\R^d}
        \sqrt{\gamma\l(x,\mu_r,\alpha_r(x)\r)}\varphi_k(x)\mathfrak{M}^{N}(dx,dr)\right|^p\right].
    \end{align*}
    Recalling $q_k=\max\{1,\|D \varphi_k\|_\infty, \|D^2\varphi_k\|_\infty\}$, we have
    \begin{align*}
        &\E^\P\left[\sup_{s\in [t,T]}|\langle \varphi_k,\mu_s^N\rangle|^p\right]
        \\
        &\leq C|\langle \varphi_k,\lambda\rangle|^p + C q_k^p\E^\P\left[ \sup_{s\in [t,T]}\left(\int_t^s|\langle \varphi_k,\mu_r^N\rangle| dr\right)^p
        \right]
        \\
        &\quad+C q_k^p\E^\P\left[ \int_t^T\langle 1, \mu_r^N\rangle^p dr \right]
        \\
        &\quad+
        C\E^\P\left[\sup_{s\in [t,T]}\left|\int_t^s\int_{\R^d}
        \sqrt{\gamma\l(x,\mu_r,\alpha_r(x)\r)}\varphi_k(x)\mathfrak{M}^{N}(dx,dr)\right|^p\right]
        .
    \end{align*}
    From Jensen's and Burkholder-Davis-Gundy's inequalities \citep[see, $e.g.$, Chapter VII, Theorem 92,][]{Dellacherie:Meyer:B}, and recalling that $\|\varphi_k\|_\infty\leq 1$, we get
    \begin{align*}
        &\E^\P\left[\sup_{s\in [t,T]}|\langle \varphi_k,\mu_s^N\rangle|^p\right]\\
        &\leq C|\langle \varphi_k,\lambda\rangle|^p +
        C q_k^p\E^\P\left[ \int_t^T|\langle \varphi_k, \mu_r^N\rangle|^p dr \right] + 
        C q_k^p\E^\P\left[ \int_t^T\langle 1, \mu_r^N\rangle^p dr \right]\\
        &\leq C q_k^p|\langle \varphi_k,\lambda\rangle|^p +
        C q_k^p\E^\P\left[ \int_t^T\sup_{s\in [t,u]}|\langle \varphi_k, X_s^N\rangle|^p dr \right] \\
        &\phantom{\leq C q_k^p|\langle \varphi_k,\lambda\rangle|^p~}
        +
        C q_k^p\E^\P\left[ \int_t^T\sup_{s\in [t,u]}\langle 1, X_s^N\rangle^p dr \right].
    \end{align*}
    Finally, multiplying by $\left(\frac{1}{2^k q_k}\right)^p$, summing over $k\in\N$ and applying the monotone convergence theorem, in addition to the fact that function equal to $1$ is in $\mathscr{F}_{\R^d}$, we have
    \beqs
        \E^\P\left[\sup_{r\in [t,T]}\d_{\R^d}(X^N_r,\0)^p\right] &\leq&C \d_{\R^d}(\lambda,\0) +
        C \E^\P\left[ \int_t^T\sup_{s\in [t,u]}\d_{\R^d}(X^N_s,\0)^p dr \right].
    \enqs
    Using Gronwall's lemma, we conclude that $\E^\P\left[\sup_{r\in [t,T]}\d_{\R^d}(X^N_r,\0)^p\right] \leq C\d_{\R^d}(\lambda,\0)^p$ for any $N\geq 1$. Applying Fatou's lemma, we obtain \eqref{Superproceq:moment_bound:mass}.
\end{proof}

\section{Dynamic programming principle}\label{Appendix:DPP}

In this section, we provide the DPP for the control problem \eqref{eq:cost_fct}-\eqref{eq:value_fct}, following the approach outlined in \citet{ElK:Tan:Capacities1,ElK:Tan:Capacities2}. This formulation relies on the so-called \textit{relaxed control} framework, introduced in \citet{ElK:Jeanblanc:exist_optim} and developped in \citet{Haussmann-Existence_Optimal_Controls}, which allows for a more flexible and comprehensive analysis of the control strategies within the given stochastic setting. We do that as in \citet[Section 5,][]{Ocello:rel_form_branching} and interpret the control set as a subset of finite measures satisfying special constraints. 

For a given $\P\in\Pc\l(\Db^d\r)$ and a stopping time $\tau$, we denote $\left(\P_\omega, \omega\in\Db^d\right)$ a regular conditional probability distribution of $\P$ given $\Fc_\tau$ \citep[see, $e.g.$, Chapter 1.1,][]{SV97}.

\subsection{Relaxed formulation}

Consider $\l[0,T\r] \times \R^d \times A$ equipped with the $\sigma$-algebra $\Bc(\l[0,T\r]) \otimes \Bc\l(\R^d\r) \otimes \Bc(A)$. Let $\mathfrak{A}^{\text{Leb}}\subseteq \Mc\l(\l[0,T\r] \times \R^d \times A\r)$ be the set of measures, whose projection on $\l[0,T\r]$ is the Lebesgue measure.
Each $\alpha \in \mathfrak{A}^{\text{Leb}}$ can be identified with its disintegration \citep[see, $e.g.$, Corollary 1.26, Chapter 1,][]{book:KALLENBERG-RM}. In particular, we have $\bar\alpha(ds, dx, da) = ds \y_s(dx) \bar \alpha_s(x,da)$, for a process $\left(\y_s(dx)\right)_s$ (resp. $\left(\bar\alpha_s(x,da)\right)_s$) taking values in the set of functions from $\l[0,T\r]$ (resp. $\l[0,T\r]\times\R^d$) to $\Mc\l(\R^d\r)$ (resp. $\Mc\l(A\r)$).

We denote $\bar\Omega := \Db^d\times\mathfrak{A}^{\text{Leb}}$. On $\bar\Omega$, let $(\mu,\beta)$ be the projection maps (or canonical processes), and $\F^{\mu,\beta}=\left\{\Fc_s^{\mu,\beta}\right\}_s$ the filtration generated by these maps, $i.e.$,
\begin{align*}
    &\Fc_s^{\mu,\beta} = \sigma\left(\mu_s(B), \beta([0,r]\times B'\times C),\phantom{\R^d}\right.
    \\
    &\qquad\qquad\qquad\qquad
    \left.
    \text{ for }s,r\in\l[0,T\r],B,B'\in\Bc\l(\R^d\r), C\in\Bc(A)\right).
\end{align*}
Moreover, define the following map
\beqs
\pi_\Ac:\Db^d\times \Ac &\to& \mathfrak{A}^{\text{Leb}}\\
(\x,\alpha)&\mapsto& ds\x_s(dx)\delta_{\alpha_s(x)}(da).
\enqs
By using controls expressed through $\pi_\Ac$, we obtain a representation of the class of controlled superprocesses, which are defined in a weak sense, within the framework of relaxed controls. This approach allows us to embed the controlled superprocesses into a broader class, ensuring that the original dynamics are preserved while providing additional flexibility in the analysis.

\begin{Definition}
    Fix $\l(t,\lambda\r) \in\l[0,T\r]\times \Nc\l[\R^d\r]$. We say that $\P\in \Pc(\bar\Omega)$ is a \emph{weak control rule}, and we denote $\P\in\Cc_{\l(t,\lambda\r)}$, if $\P(\mu_t=\lambda)=1$, there exists $\alpha^\P\in\Ac$ such that $\P\left(\pi_\Ac\left(\mu,\alpha^\P\right)= \beta\right)=1$, and the process
    \beqs
    M^{F_\varphi}_s = F_\varphi\l(\mu_s\r) - \int_t^s \int_{\R^d\times A} \Lc F_\varphi (x,\mu_r,a)\beta_s(x,da)\mu_r(dx)dr
    \enqs
    is a $\left(\P,\F^{\mu,\beta}\right)$-martingale with quadratic variation
    \beqs
    \left[M^{F_\varphi}\right]_s = \int_t^s  \left(F'_\varphi(\mu_r)\right)^2\int_{\R^d\times A}\gamma(x,\mu_r,a)\beta_s(x,da)\varphi^2(x)\mu_r(dx)dr,
    \enqs
    for any $F\in C^{2}_b\l(\R\r)$, $\varphi\in C^2_{b}\l(\R^d\r)$, and $s\geq t$.
\end{Definition}

It is clear that each element of $\Cc_{\l(t,\lambda\r)}$ can be identified to an element of $\Rc_{\l(t,\lambda\r)}$, and viceversa. With abuse of notation, we write $J(t,\lambda,\P)$ for $\P\in \Cc_{\l(t,\lambda\r)}$ to denote $J(t,\lambda,\alpha^\P)$. With this description, we have
\beqs
v\l(t,\lambda\r) =\inf_{\alpha\in\Ac}J(t,\lambda,\alpha)= \inf_{\P\in\Cc_{\l(t,\lambda\r)}}J(t,\lambda,\P).
\enqs

In this framework, we can consider the notion of conditioning as well as concatenation on $\bar\Omega$. For $(t,\bar w)\in \l[0,T\r]\times \bar\Omega$, we denote
\beqs
    \mathfrak{P}^t_{\bar w} &:=& \left\{\bar \omega: \mu_t(\bar \omega)=\mu_t(\bar w)\right\},\\
    \mathfrak{P}_{t,\bar w} &:=&
    \l\{
        \bar \omega: \left(\mu_s, \mathfrak{M}_s(\phi)\right)(\bar \omega)= \left(\mu_s, \mathfrak{M}_s(\phi)\right)(\bar w),
        \phantom{\R^d}
        \r.
        \\
        &&\qquad\qquad\qquad\qquad\l.
    \text{ for }s\in \l[0,T\r], \phi\in C_b\left(\l[0,T\r]\times \R^d\times A \right)\r\},
\enqs
where
\beqs
\mathfrak{M}_s(\phi):= \int_0^s\int_{\R^d\times A} \phi(s,x,a)\beta(ds,dx,da).
\enqs
Then, for all $\bar \omega\in \mathfrak{P}^t_{\bar w}$, we define the concatenated path 
$\bar w \otimes_t\bar \omega$ by
\beqs
\left(\mu_s, \mathfrak{M}_s(\phi)\right)(\bar w \otimes_t\bar \omega) = \begin{cases}
    \left(\mu_s, \mathfrak{M}_s(\phi)\right)(\bar w), & \text{ for } s\in [0,t),\\
    \left(\mu_s, \mathfrak{M}_s(\phi) - \mathfrak{M}_t(\phi)\right)(\bar \omega)&
    \\
    \qquad\qquad+ \left(\mu_s,\mathfrak{M}_t(\phi)\right)(\bar w), & \text{ for } s\in [t,T],
\end{cases}
\enqs
for all $\phi\in C_b\left(\l[0,T\r]\times \R^d\times A \right)$.

Fix $\P\in \Pc(\bar \Omega)$, and $\tau$ a $\F^{\mu,\beta}$-stopping time. From \citet[Proposition 1.9, Chapter 1,][]{Yong-Zhou-StochasticControls},
there is a family of regular conditional probability distribution (r.c.p.d.) $\left(\P_{\bar\omega}\right)_{\bar\omega\in\bar\Omega}$ w.r.t. $\Fc^{\mu,\beta}_\tau$ such that the $\Fc^{\mu,\beta}_\tau$-measurable probability kernel $\left(\P_{\bar\omega}\right)_{\bar\omega\in\bar\Omega}$ satisfies 
\beqs
\P_{\bar\omega}\left(\mathfrak{P}_{\tau(\bar\omega),\bar\omega}\right)=1\quad \text{ for }\P-\text{a.e. }\bar\omega\in \bar\Omega.
\enqs
On the other hand, take a probability measure $\P$ defined on $\left(\bar \Omega, \Fc^{\mu,\beta}_\tau\right)$ and a family of probability measures $\left(\Q_{\bar\omega}\right)_{\bar\omega\in\bar\Omega}$ such that $\bar\omega\mapsto \Q_{\bar\omega}$ is $\Fc^{\mu,\beta}_\tau$-measurable and 
\beqs
\Q_{\bar\omega}\left(\mathfrak{P}^{\tau(\bar\omega)}_{\bar\omega}\right)=1 \quad \text{ for }\P-\text{a.e. }\bar\omega\in \bar\Omega.
\enqs
There is a unique concatenated probability measure that we denote $\P\otimes_\tau\Q_\cdot$ defined by
\beqs
\P\otimes_\tau\Q_\cdot(C)&:=&\int_{\bar\Omega} \P(d\bar w)
    \int_{\bar\Omega}\1_C\left(\bar w \otimes_{\tau(\bar w)}\bar \omega\right)\Q_{\bar w}(d\bar \omega),
 \quad \text{ for }C\in \Fc^{\mu,\beta}_T.
\enqs

\subsection{Measurable selection and DPP}

This weak formulation has the advantage of simplifying the proof of the DPP. We follow the path detailed in \citet{ElK:Tan:Capacities1} and \citet{ElK:Tan:Capacities2}, which clarify \citet[Chapter 7,][]{bertsekas1996stochastic} in the context of stochastic control theory, generalizing it to our setting. In particular, to reach the DPP, as in \citet[Theorem 4.10,][]{ElK:Tan:Capacities1} and \citet[Theorem 3.1,][]{ElK:Tan:Capacities2}, we need to show that the setting so far presented satisfies \citet[Assumption 2.2,][]{ElK:Tan:Capacities2}, which in our setting reads as follows.

It is clear that the \textit{full} generator $\G$ as defined in \eqref{Superproceq:def_full_generator} is countably generated. Combining Proposition \ref{Prop:existence} with the weak formulation description, $\Cc_{\l(t,\lambda\r)}$ is nonempty for all $\l(t,\lambda\r)\in\l[0,T\r]\times \Mc\l(\R^d\r)$. Fix $\l(t,\lambda\r)\in\l[0,T\r]\times \Mc\l(\R^d\r)$, $\P\in\Cc_{\l(t,\lambda\r)}$ and $\tau$ a $\F^{\mu,\beta}$-stopping time taking value in $[t,T]$. Using \citet[Lemma 3.2,][]{ElK:Tan:Capacities2} and \citet[Lemma 3.3,][]{ElK:Tan:Capacities2}, we obtained
\begin{itemize}
    \item \emph{Stability by conditioning}: There is a family of r.c.p.d. $\left(\P_{\bar\omega}\right)_{\bar\omega\in\bar\Omega}$ w.r.t. $\Fc^{\mu,\beta}_\tau$ such that $\P_{\bar\omega}\in \Cc_{\left(\tau(\bar\omega),\mu(\bar\omega)\right)}$ for $\P$-a.e. $\bar\omega\in\bar\Omega$.
    \item \emph{Stability by concatenation}: Let $\left(\Q_{\bar\omega}\right)_{\bar\omega\in\bar\Omega}$ be a probability kernel from $\Fc^{\mu,\beta}_\tau$ into $\left(\bar\Omega,\Fc^{\mu,\beta}_T\right)$ such that $\bar\omega\mapsto \Q_{\bar\omega}$ is $\Fc^{\mu,\beta}_\tau$-measurable, and $\Q_{\bar\omega}\in \Cc_{\left(\tau(\bar\omega),\mu(\bar\omega)\right)}$ for $\P$-a.e. $\bar\omega\in\bar\Omega$. Then, $\P\otimes_\tau \Q_\cdot\in\Cc_{\l(t,\lambda\r)}$.
\end{itemize}

These two conditions are those of \citet[Assumption 2.2,][]{ElK:Tan:Capacities2}. This allows us to prove the following DPP.

\begin{Theorem}\label{SuperprocTheorem:DPP}
    We have
    \begin{equation}\label{SuperprocDPP}
            \begin{split}
          v\l(t,\lambda\r) = \inf_{\P\in\Cc_{\l(t,\lambda\r)}}\E^{\P}\left[
          \int_t^\tau \int_{\R^d\times A} \psi(x,\mu_s,a)\beta_s(x,da)\mu_s(dx)ds + v(\tau,\mu_\tau)\right]\\
          = \inf_{\alpha\in\Ac}\E^{\P^{t,\lambda,\alpha}}\left[
          \int_t^\tau \int_{\R^d} \psi(x,\mu_s,\alpha_s(x))\mu_s(dx)ds + v(\tau,\mu_\tau)\right],\qquad
        \end{split}
    \end{equation}
    for any $\l(t,\lambda\r) \in \l[0,T\r]\times \Mc\l(\R^d\r)$, and $\tau$ stopping time taking value in $[t,T]$.
\end{Theorem}

\begin{proof}
    Fix $\l(t,\lambda\r) \in \l[0,T\r]\times \Mc\l(\R^d\r)$, and $\tau$ to be a stopping time taking values in $[t, T]$. We have that the cost function \eqref{eq:cost_fct} is continuous, thus a fortiori upper semi-analytic. Following the stability by conditioning, for any $\P\in \Cc_{\l(t,\lambda\r)}$, there is $\left(\P_{\bar\omega}\right)_{\bar\omega\in\bar\Omega}$ a family of r.c.p.d. w.r.t. $\Fc^{\mu,\beta}_\tau$ such that $\P_{\bar\omega}\in \Cc_{\left(\tau(\bar\omega),\mu(\bar\omega)\right)}$ for $\P$-a.e. $\bar\omega\in\bar\Omega$. Therefore, we get 
    \begin{align*}
        &J\left(\tau(\bar\omega),\mu_{\tau(\bar\omega)}(\bar\omega),\P_{\bar\omega}\right)
        \\
        &=
        \E^{\P_{\bar\omega}} \left[
        \int_\tau^T \int_{\R^d\times A} \psi(x,\mu_s,a)\beta_s(x,da)\mu_s(dx)ds + \Psi(\mu_T)\right],
        \\
        &\qquad\qquad\qquad\qquad\qquad\qquad\qquad\qquad\qquad\qquad\quad \text{ for }\P\text{-a.e. }\bar\omega\in\bar\Omega.
    \end{align*}
    Since, by definition, $v(\tau(\bar\omega),\mu_{\tau(\bar\omega)}(\bar\omega))\leq J\left(\tau(\bar\omega),\mu_{\tau(\bar\omega)}(\bar\omega),\P_{\bar\omega}\right)$, it follows from the tower property of conditional expectations that
    \begin{align*}
        J(t,\lambda,\P) &= \int_{\bar \Omega}\left(
            J\l(
                \tau(\bar\omega), \mu_{\tau(\bar\omega)}(\bar\omega), \P_{\bar\omega}
            \r)
        \phantom{\int_t^\tau}
        \right.
        \\
        &
        \qquad\qquad
        \left.
        + \int_t^\tau \int_{\R^d\times A} \psi(x,\mu_s,a)\beta_s(x,da)\mu_s(dx)ds\right)\P(d\bar \omega)
        \\
        &\geq
    \E^{\P}\left[v(\tau,\mu_\tau) +
        \int_t^\tau \int_{\R^d\times A} \psi(x,\mu_s,a)\beta_s(x,da)\mu_s(dx)ds \right],
    \end{align*}
    which, 
    by the arbitrariness of $\P$, provides
    \begin{align*}
        v\l(t,\lambda\r) \geq \displaystyle\inf_{\P\in\Cc_{\l(t,\lambda\r)}}\E^{\P}\left[
        \int_t^\tau \int_{\R^d\times A} \psi(x,\mu_s,a)\beta_s(x,da)\mu_s(dx)ds + v(\tau,\mu_\tau)\right].
    \end{align*}

    We now turn to the reverse inequality. Fix some arbitrary $\P\in\Cc_{\l(t,\lambda\r)}$ and $\eps>0$. Consider the set $\Cc^\eps_{(t',\lambda')}$ defined as follows
    \beqs
    \Cc^\eps_{(t',\lambda')}:= \left\{\Q\in \Cc_{(t',\lambda')}:v(t',\lambda')+\eps \geq J(t',\lambda',\Q)\right\}, \qquad\qquad \qquad\qquad 
    \\
    \text{ for }(t',\lambda')\in \l[0,T\r]\times \Mc\l(\R^d\r).
    \enqs
    From the \citet[Proposition 2.21,][]{ElK:Tan:Capacities1}, there exists a family of probability $\left(\Q^\eps_{\bar\omega}\right)_{\bar\omega\in\bar\Omega}$ from $\Fc^{\mu,\beta}_\tau$ into $\left(\bar\Omega,\Fc^{\mu,\beta}_T\right)$ such that $\bar\omega\mapsto \Q^\eps_{\bar\omega}$ is $\Fc^{\mu,\beta}_\tau$-measurable, and $\Q^\eps_{\bar\omega}\in \Cc^\eps_{\left(\tau(\bar\omega),\mu(\bar\omega)\right)}$ for $\P$-a.e. $\bar\omega\in\bar\Omega$. Then, $\P\otimes_\tau \Q^\eps_\cdot\in\Cc_{\l(t,\lambda\r)}$ by the stability by concatenation condition. This implies that    
    \beqs
    &&J(t,\lambda,\P\otimes_\tau \Q^\eps_\cdot)
    \\
    &&= \E^{\P\otimes_\tau \Q^\eps_\cdot}
    \left[
        \int_t^T \int_{\R^d\times A} \psi(x,\mu_s,a)\beta_s(x,da)\mu_s(dx)ds + \Psi(\mu_T)
        \right]
    \\
    &&= \int
        \Bigg(
            \int_t^{\tau(\bar\omega)} \int_{\R^d\times A} \psi(x,\mu_s,a)\beta_s(x,da)\mu_s(dx)ds \\
    &&\quad \phantom{\int \Bigg( + }
            + \E^{\Q^\eps_{\bar\omega}}
                \left[\int_\tau^T \int_{\R^d\times A} \psi(x,\mu_s,a)\beta_s(x,da)\mu_s(dx)ds + \Psi(\mu_T)\right]
        \Bigg) \P(d\bar\omega)\\
    &&= \E^{\P}\left[
        \int_t^{\tau(\bar\omega)}\int_{\R^d\times A} \psi(x,\mu_s,a)\beta_s(x,da)\mu_s(dx)ds + J\left(\tau(\bar\omega),\mu_{\tau(\bar\omega)}(\bar\omega),\Q^\eps_{\bar\omega}\right)\right]\\
    &&\leq \E^{\P}\left[
        \int_t^{\tau}\int_{\R^d\times A} \psi(x,\mu_s,a)\beta_s(x,da)\mu_s(dx)ds + v\left(\tau,\mu_\tau \right) \right] + \eps.
    \enqs
    From the arbitrariness of $\P\in \Cc_{\l(t,\lambda\r)}$ and $\eps > 0$, we obtain the inequality
    \beqs
    v\l(t,\lambda\r) \leq \inf_{\P\in \Cc_{\l(t,\lambda\r)}}\E^{\P}\left[
        \int_t^\tau \int_{\R^d\times A} \psi(x,\mu_s,a)\beta_s(x,da)\mu_s(dx)ds + v(\tau,\mu_\tau)\right].
    \enqs
\end{proof}

\vspace{2mm}
\noindent \textbf{Acknowledgements.} I gratefully acknowledge my Ph.D. supervisor Idris Kharroubi for supervising this work.

\bibliographystyle{plainnat}
\bibliography{main.bib}

% \begin{thebibliography}{10}  

% \end{thebibliography}

\end{document}